\documentclass[reqno]{amsart}
\usepackage{amsfonts}
\usepackage{mathrsfs}
\usepackage{cite}
\usepackage[dvips]{color}
\allowdisplaybreaks
\date{\today}

\theoremstyle{plain}
\newtheorem{thm}{Theorem}[section]

\newtheorem{lem}[thm]{Lemma}
\newtheorem{prop}[thm]{Proposition}

\theoremstyle{definition}

\theoremstyle{remark}
\newtheorem{rem}{Remark}[section]
\numberwithin{equation}{section}

\renewcommand{\u}{{\mathbf u}}

\renewcommand{\t}{{\mathbf{B}}}
\renewcommand{\H}{\mathbf{H}}

\newcommand{\R}{{\mathbb R}}
\newcommand{\w}{{\mathbf w}}

\newcommand{\dv}{{\rm div }}
\newcommand{\cu}{{\rm curl\, }}
\newcommand{\E}{{\mathcal E}}

\begin{document}

\title[Low Mach number limit for the full compressible MHD equations]
{Low Mach number limit for the  full compressible
magnetohydrodynamic equations with general initial data}

\author{Song Jiang}
\address{LCP, Institute of Applied Physics and Computational Mathematics, P.O.
 Box 8009, Beijing 100088, P.R. China}
 \email{jiang@iapcm.ac.cn}

\author{Qiangchang Ju}
\address{Institute of Applied Physics and Computational Mathematics, P.O.
 Box 8009-28, Beijing 100088, P.R. China}
 \email{qiangchang\_ju@yahoo.com}

 \author[Fucai Li]{Fucai Li$^*$}
\address{Department of Mathematics, Nanjing University, Nanjing
 210093, P.R. China}
 \email{fli@nju.edu.cn}
 \thanks{$^*$Corresponding author}

\author{Zhouping Xin}
\address{The Institute of Mathematical Sciences, The Chinese University of Hong Kong,
Shatin, NT, Hong Kong}
 \email{zpxin@ims.cuhk.edu.hk}



\begin{abstract}
The low Mach number limit for the  full compressible
magnetohydrodynamic equations with general initial data is
rigorously justified in the whole space $\mathbb{R}^3$. First, the uniform-in-Mach-number estimates of the solutions
in a Sobolev space are established on a finite time interval independent of
the Mach number. Then the low Mach number limit is proved by combining these
uniform estimate with a theorem due to M\'etiver and Schochet
[Arch. Ration. Mech. Anal. 158 (2001), 61-90] for the Euler equations
that gives the local energy decay of the acoustic wave equations.

\end{abstract}

\maketitle


\section{Introduction}
In this paper we study the low Mach number limit of local smooth solutions to the
following full compressible magnetohydrodynamic (MHD) equations
with general initial data in the whole space $\mathbb{R}^3$ (see \cite{JT64,KL,LL,PD}):
\begin{align}
&\partial_t\rho +\dv(\rho\u)=0, \label{haa} \\
&\partial_t(\rho\u)+\dv\left(\rho\u\otimes\u\right)+ {\nabla
P}
  =\frac{1}{4\pi}(\cu \H)\times \H+\dv\Psi(\u), \label{hab} \\
&\partial_t\H-\cu(\u\times\H)=-\cu(\nu\,\cu\H),\quad
\dv\H=0, \label{hac}\\
&\partial_t\E+\dv\left(\u(\E'+P)\right)
=\frac{1}{4\pi}\dv((\u\times\H)\times\H)\nonumber\\
& \qquad \qquad \qquad\qquad\qquad\, \, +\dv\Big(\frac{\nu}{4\pi}
\H\times(\cu\H)+ \u\Psi(\u)+\kappa\nabla\theta\Big).
 \label{had}
\end{align}
Here   the unknowns $\rho $, $\u=(u_1,u_2,u_3)\in \R^3$ , $\H=(H_1,H_2,H_3)\in \R^3$,
and $\theta$  denote the density, velocity, magnetic field, and temperature, respectively;
$\Psi(\u)$
is the viscous stress tensor given by
\begin{equation*}
\Psi(\u)=2\mu \mathbb{D}(\u)+\lambda\dv\u \;\mathbf{I}_3
\end{equation*}
with $\mathbb{D}(\u)=(\nabla\u+\nabla\u^\top)/2$,
 $\mathbf{I}_3$  the $3\times 3$ identity matrix,
and $\nabla \u^\top$   the transpose of the matrix
$\nabla \u$; $\E$ is the total energy given by $\E=\E'+|\H|^2/({8\pi})$ and
$\E'=\rho\left(e+|\u|^2/2 \right)$ with $e$ being the internal
energy, $\rho|\u|^2/2$ the kinetic energy, and
$|\H|^2/({8\pi})$ the magnetic energy. The viscosity
coefficients $\lambda$ and $\mu$ of the flow satisfy
 $\mu>0$ and $2\mu+3\lambda>0$. The parameter $\nu>0$ is the magnetic diffusion
coefficient of the magnetic field    and $\kappa>0$   the heat
conductivity. For simplicity, we assume that $\mu,\lambda,\nu$ and
$\kappa$ are constants. The equations of state $P=P(\rho,\theta)$ and
$e=e(\rho,\theta)$ relate the pressure $P$ and the internal energy
$e$ to the density $\rho$ and the temperature $\theta$ of the flow.

Multiplying  \eqref{hab} by $\u$ and \eqref{hac} by $\H/({4\pi})$
and summing over, one finds that
\begin{align}\label{haaz}
&\frac{\rm d}{{\rm d}t}\Big(\frac{1}{2}\rho|\u|^2+\frac{1}{8\pi}|\H|^2\Big)
+\frac{1}{2}\dv\left(\rho|\u|^2\u\right)+\nabla P\cdot\u \nonumber\\
&\quad  =\dv\Psi\cdot
\u+\frac{1}{4\pi}(\cu\H)\times\H\cdot\u
+\frac{1}{4\pi}\cu(\u\times\H)\cdot\H\nonumber\\
& \qquad  -\frac{\nu}{4\pi}\cu(\cu\H)\cdot\H.
\end{align}
Due to the identities
\begin{eqnarray} &&
 \dv(\H\times(\cu\H))  =|\cu\H|^2-\cu(\cu\H)\cdot\H , \nonumber \\
 && \dv((\u\times\H)\times\H)   =(\cu\H)\times\H\cdot\u+\cu(\u\times\H)\cdot\H, \label{nae}
\end{eqnarray}
one can subtract \eqref{haaz} from \eqref{had} to rewrite
the energy equation \eqref{had} in terms of the internal energy as
\begin{equation}\label{hagg}
\partial_t (\rho e)+\dv(\rho\u e)+(\dv\u)P=\frac{\nu}{4\pi}|\cu\H|^2+\Psi(\u):\nabla\u+\kappa \Delta \theta,
\end{equation}
where $\Psi(\u):\nabla\u$ denotes the scalar product of two matrices:
\begin{equation*}
\Psi(\u):\nabla\u=\sum^3_{i,j=1}\frac{\mu}{2}\left(\frac{\partial
u^i}{\partial x_j} +\frac{\partial u^j}{\partial
x_i}\right)^2+\lambda|\dv\u|^2=
2\mu|\mathbb{D}(\u)|^2+\lambda|\mbox{tr}\mathbb{D}(\u)|^2.
\end{equation*}

To establish  the low Mach number limit for the system
\eqref{haa}--\eqref{hac}, \eqref{hagg}, in this paper we shall
focus on the ionized fluids obeying the following perfect gas relations
\begin{align}
P=\mathfrak{R}\rho \theta,\quad e=c_V\theta,\label{hpg}
\end{align}
where the parameters $\mathfrak{R}>0$ and $ c_V\!>\!0$ are the gas constant and
the heat capacity at constant volume, respectively, which will be assumed to be one for simplicity
of the presentation. We also ignore  the coefficient $1/(4\pi)$ in the magnetic field.

Let $\epsilon$ be the Mach number, which is a dimensionless number. Consider
the system \eqref{haa}--\eqref{hac}, \eqref{hagg} in the physical regime:
\begin{gather*}
 P \sim P_0+O(\epsilon), \quad \u\sim O(\epsilon), \quad \H\sim
 O(\epsilon),\quad \nabla\theta\sim O(1),
\end{gather*}
where $P_0>0$ is a certain given constant which is normalized to be $P_0= 1$.
 Thus we consider the case when the pressure $P$ is a small perturbation of
 the given state $1$, while the temperature $\theta$ has a finite variation.
As in \cite{A06}, we introduce the following
transformation to ensure positivity of $P $ and $\theta$
\begin{align}
 & \  P(x,t)= e^{\epsilon p^\epsilon(x,\epsilon t)}, \quad
 \theta(x,t)=e^{\theta^\epsilon(x,\epsilon t)}, 
\label{transa}
\end{align}
where a longer time scale $t=\tau/\epsilon$ (still denote $\tau$  by
$t$ later for simplicity)  is introduced in order to seize the
evolution of the fluctuations. Note that \eqref{hpg} and
\eqref{transa} imply that $ \rho(x,t)=e^{\epsilon
p^\epsilon(x,\epsilon t)-\theta^\epsilon(x,\epsilon t)}$ since $
\mathfrak{R}\equiv c_V\equiv 1$. Set
\begin{align}\label{transb}
  {\H} (x,t)=\epsilon \H^\epsilon(x,\epsilon t), \quad
 {\u} (x,t)=\epsilon \u^\epsilon(x,\epsilon t),
  \end{align}
  and  \begin{align}\nonumber
   \mu=\epsilon \mu^\epsilon, \quad \lambda=\epsilon  \lambda^\epsilon,
   \quad \nu=\epsilon \nu^\epsilon,  \quad \kappa=\epsilon  \kappa^\epsilon.
 \end{align}

Under these changes of variables and coefficients, the system,
\eqref{haa}--\eqref{hac}, \eqref{hagg} with \eqref{hpg}, takes
 the following equivalent form:
 \begin{align}
   &\partial_t p^\epsilon +(\u^\epsilon \cdot\nabla)p^\epsilon
   +\frac{1}{\epsilon}\dv(2\u^\epsilon-\kappa^\epsilon
   e^{-\epsilon p^\epsilon+\theta^\epsilon}\nabla \theta^\epsilon)\nonumber\\
 & \quad \quad =\epsilon e^{-\epsilon p^\epsilon}
[\nu^\epsilon|\cu\H^\epsilon|^2
+\Psi(\u^\epsilon):\nabla\u^\epsilon]+\kappa^\epsilon
e^{-\epsilon p^\epsilon+\theta^\epsilon}\nabla p^\epsilon \cdot \nabla  \theta^\epsilon,
\label{hnaaa} \\
&e^{-\theta^\epsilon}[\partial_t\u^\epsilon+(\u^\epsilon\cdot\nabla)\u^\epsilon]
+\frac{\nabla p^\epsilon}{\epsilon}
  =e^{- \epsilon p^\epsilon}[(\cu \H^\epsilon)\times \H^\epsilon
  +\dv\Psi^\epsilon(\u^\epsilon)], \label{hnabb} \\
&\partial_t\H^\epsilon-\cu(\u^\epsilon\times\H^\epsilon)
-\nu^\epsilon\Delta \H^\epsilon=0,\quad
\dv\H^\epsilon=0,\label{hnacc} \\
&\partial_t \theta^\epsilon +(\u^\epsilon\cdot\nabla)
\theta^\epsilon+ \dv\u^\epsilon\nonumber\\
& \quad  \quad =\epsilon^2e^{-\epsilon p^\epsilon}[\nu^\epsilon|\cu\H^\epsilon|^2
+\Psi^\epsilon(\u^\epsilon):\nabla\u^\epsilon]+\kappa^\epsilon
e^{-\epsilon p^\epsilon}\dv (e^{\theta^\epsilon}\nabla  \theta^\epsilon),
  \label{hnadd}
\end{align}
where
$\Psi^\epsilon(\u^\epsilon)=2\mu^\epsilon \mathbb{D}(\u^\epsilon)+\lambda^\epsilon\dv\u^\epsilon\,\mathbf{I}_3$,
and the identity $\cu (\cu \H^\epsilon)=\nabla \dv\H^\epsilon-\Delta \H^\epsilon$ and the constraint
that $\dv\H^\epsilon=0$ have been used.

 We shall study the limit as $\epsilon\to 0$ of  solutions to the system
 \eqref{hnaaa}--\eqref{hnadd}.
Formally, as $\epsilon$ goes to zero, if the
sequence $(p^\epsilon, \u^\epsilon,\H^\epsilon, \theta^\epsilon)$ converges
strongly to a limit $(1,\w,\t,\vartheta)$ in some sense, and
$(\mu^\epsilon,\lambda^\epsilon,\nu^\epsilon,\kappa^\epsilon)$
converges to a constant vector
$(\bar\mu,\bar\lambda,\bar\nu,\bar\kappa)$, then taking the
limit to \eqref{hnaaa}--\eqref{hnadd}, we have
\begin{align}
&\dv(2\w -\bar{\kappa}\, e^\vartheta\nabla \vartheta)=0, \label{hna1} \\
&e^{-\vartheta}[\partial_t\w+(\w\cdot\nabla)\w]+\nabla \pi
  =(\cu \t)\times \t+\dv\Phi(\w), \label{hna2} \\
&\partial_t\t -\cu(\w\times\t)-\bar \nu\Delta\t=0,\quad
\dv\t=0,    \label{hna3} \\
&\partial_t\vartheta  +(\w\cdot\nabla)\vartheta +\dv \w=\bar \kappa  \, \dv (e^{\vartheta}\nabla  \vartheta),
 \label{hna4}
\end{align}
with some function $\pi$, where $\Phi(\w)$ is defined by
\begin{equation}\nonumber 
\Phi(\w)=2\bar \mu \mathbb{D}(\w)+\bar\lambda\dv\w \,\mathbf{I}_3.
\end{equation}

The purpose of this paper is to establish the above limit process rigorously with general initial data. For this purpose, we
supplement the system \eqref{hnaaa}--\eqref{hnadd} with the following initial conditions
\begin{align}
(p^\epsilon,\u^\epsilon,\H^\epsilon, \theta^\epsilon)|_{t=0}
=(p^\epsilon_{\rm in}(x),\u^\epsilon_{\rm in}(x),\H^\epsilon_{\rm in}(x),
\theta^\epsilon_{\rm in}(x)), \quad x \in \mathbb{R}^3. \label{hnas}
\end{align}

For simplicity of presentation, we shall assume that $\mu^\epsilon
\equiv \bar\mu>0$, $\nu^\epsilon \equiv \bar\nu>0$, $\kappa^\epsilon
\equiv \bar\kappa>0$, and $\lambda^\epsilon \equiv \bar\lambda$. The
general case $\mu^\epsilon\rightarrow \bar\mu>0$,
$\nu^\epsilon\rightarrow\bar\nu>0$,
$\kappa^\epsilon\rightarrow\bar\kappa>0$ and
$\lambda^\epsilon\rightarrow\bar\lambda$ simultaneously as $\epsilon
\rightarrow 0$ can be treated by slightly modifying the arguments presented in this paper.

As in \cite{A06}, we will use the notation
$\|v\|_{H_\eta^\sigma}:=\|v\|_{H^{\sigma-1}}+\eta\|v\|_{H^\sigma}$
for any $\sigma\in \mathbb{R}$ and $\eta\geq 0$. For each $\epsilon>0$, $t\geq 0$ and $s\geq 0$,
we will also use the following norm:
 \begin{align*}\nonumber
&\|(p^\epsilon,\u^\epsilon, \H^\epsilon,\theta^\epsilon-\bar \theta)(t)\|_{s,\epsilon} \nonumber\\
 &\qquad :=  \sup_{\tau\in [0,t]}\big\{\|(p^\epsilon,\u^\epsilon,\H^\epsilon)(\tau)\|_{H^{s}}
 +\|(\epsilon p^\epsilon,\epsilon\u^\epsilon, \epsilon\H^\epsilon,
 \theta^\epsilon-\bar \theta)(\tau)\|_{H_\epsilon^{s+2}}\big\}\nonumber\\
  & \qquad \ \ \quad  + \bigg\{\int^t_0\big[\|\nabla(p^\epsilon,\u^\epsilon,\H^\epsilon)\|^2_{H^s}
  +\|\nabla (\epsilon\u^\epsilon,\epsilon H^\epsilon,\theta^\epsilon)\|^2_{H_\epsilon^{s+2}}\big](\tau){\rm d}\tau\bigg\}^{1/2}.
 \end{align*}

Then, the main result of this paper reads as follows.
\begin{thm}\label{main}
 Let $s\geq  4$ be an integer. Assume that the initial data
$ (p^\epsilon_{\rm in},\u^\epsilon_{\rm in},\H^\epsilon_{\rm in}, \linebreak \theta^\epsilon_{\rm in})$ satisfy
\begin{align}\label{initial}
\|(p^\epsilon_{\rm in},\u^\epsilon_{\rm in},\H^\epsilon_{\rm in})\|_{H^{s}}
 +\|(\epsilon p^\epsilon_{\rm in},\epsilon\u^\epsilon_{\rm in}, \epsilon\H^\epsilon_{\rm in},
 \theta^\epsilon_{\rm in}-\bar \theta)\|_{H_\epsilon^{s+2}} \leq L_0
\end{align}
 for all $\epsilon \in (0,1]$ and two given positive constants $\bar \theta$ and $L_0$.
 Then  there exist positive constants $T_0$ and $\epsilon_0<1$, depending only on $L_0$ and $\bar \theta$, such
 that  the Cauchy problem \eqref{hnaaa}--\eqref{hnadd}, \eqref{hnas} has a unique
solution $(p^\epsilon,\u^\epsilon,\H^\epsilon, \theta^\epsilon)$ satisfying
\begin{align}\label{conda}
 \|(p^\epsilon,\u^\epsilon,\H^\epsilon, \theta^\epsilon-\bar \theta)(t)\|_{s,\epsilon}\leq L,
 \quad \forall \,\, t\in [0,T_0], \ \forall\,\epsilon \in (0,\epsilon_0],
\end{align}
where $L$ depends only on $L_0$, $\bar \theta$ and $T_0$.
Moreover, assume further that  the initial data satisfy the following conditions
 \begin{align}
& |\theta^\epsilon_0(x)-\bar{\theta} |\leq   {N}_0 |x|^{-1-\zeta}, \quad
 |\nabla \theta^\epsilon_0(x)|\leq  N_0 |x|^{-2-\zeta}, \quad \forall  \, \epsilon  \in (0,1],\label{dacay}\\
& \big(p^\epsilon_{\rm in}, \u^\epsilon_{\rm in},
\H^\epsilon_{\rm in}, \theta^\epsilon_{\rm in}-\bar \theta\big)\rightarrow
(0,\u_0,\t_0, \vartheta_0-\bar \theta)\ \ \mbox{ in }\ H^{s}(\mathbb{R}^3)  \label{conver}
\end{align}
 as $\epsilon\rightarrow 0$, where $N_0$ and $\zeta$ are fixed positive constants.
  Then the solution sequence $(p^\epsilon,\u^\epsilon,\H^\epsilon,$ $ \theta^\epsilon-\bar\theta)$ converges weakly  in
  $L^\infty(0,T_0; H^s(\R^3))$
and strongly in $L^2(0,T_0;$    $H^{s_2}_{\mathrm{loc}}(\R^3))$ for all  $0\leq s_2<s$
to the limit $(0,\w,\t, \vartheta-\bar\theta)$, where $(\w,\t, \vartheta)$
 satisfies  the system  \eqref{hna1}--\eqref{hna4} with initial
data $(\w,\t, \vartheta)|_{t=0} =(\w_0,\t_0, \vartheta_0)$, where $\w_0$ is determined by
\begin{align}\label{intlimit}
   \dv(2\w_0 -\bar{\kappa}\, e^{\vartheta_0}\nabla \vartheta_0)=0,
   \quad \cu(e^{-\vartheta_0}\w_0)= \cu(e^{-\vartheta_0}\u_0).
\end{align}
\end{thm}

We now give some comments on the proof of  Theorem \ref{main}. The
key point in the proof  is to establish the uniform estimates in Sobolev norms
for the acoustic components of solutions, which are propagated by the wave
equations whose coefficients are functions of the temperature. Our main strategy
is to bound the norm of $(\nabla p^\epsilon, \dv \u^\epsilon)$ in terms of the norms
of $(\epsilon \partial_t) (p^\epsilon,\u^\epsilon, \H^\epsilon)$ and
 $(\epsilon p^\epsilon,\epsilon \u^\epsilon,\epsilon \H^\epsilon, \theta^\epsilon )$
 through the density and the momentum equations.
This approach is motivated by the previous works on the compressible Navier-Stokes
equations due to Alazard
\cite{A06}, and  Levermore, Sun and Trivisa \cite{LST}.
It should be pointed out that the
analysis for \eqref{hnaaa}--\eqref{hnadd} is very complicated   due to the strong coupling of the hydrodynamic motion and
the magnetic fields. More efforts
should be paid on the estimates involving these coupling terms, in particular,
on the estimate of higher order spatial derivatives. We shall
exploit the special structure of the system to obtain the tamed
estimate on higher order derivatives, so that we can close our
estimates on the uniform boundedness of the solutions.
 Once the uniform boundedness of the solutions has been established,
 one can obtain the convergence result in Theorem \ref{main}
by applying the compactness arguments and the dispersive estimates
on the acoustic wave equations in the whole space developed in \cite{MS01}.

\begin{rem}
 If we take $\H=0$ in \eqref{hnaaa}--\eqref{hnadd}, the system \eqref{hnaaa}--\eqref{hnadd}
 reduces to the full compressible Navier-Stokes  equations.
 The estimates obtained in the present paper provide a simpler proof of a somewhat weaker result than
 that was proven in \cite{A06}. Specifically, by giving up
the uniformity in the viscosity coefficients and heat conductivity parameter, we can eliminate the
necessity of separating estimates on  high and low frequencies, leading thus to a simpler proof.
\end{rem}

\begin{rem}
 The positivity of the coefficients  $\mu$, $\nu$ and $\kappa$ plays a fundamental role  in the
 proof of Theorem \ref{main}. The arguments given in this paper cannot be applied to the case
 when one of them disappears. In fact, in the case of $\mu=\nu=\kappa=0$,  the terms
$(\cu{\H^\epsilon}) \times {\H^\epsilon}$ in the momentum equations and
$\cu({\u^\epsilon} \times {\H^\epsilon})$ in the magnetic field
equation change basically the structure of the system. Recently, Jiang, Ju and Li \cite{JJL4}
  have studied the incompressible limit of the compressible non-isentropic
  ideal MHD equations with general initial data in the whole space
$\mathbb{R}^d$ ($d=2,3$) when the  initial data belong to $H^s(\mathbb{R}^d)$ with $s\, (\geq 4)$
being an \emph{even} integer. We emphasize that
the restriction on the Sobolev index $s$ to be even plays a
crucial role in the proof since in this case the nonstandard highest order
derivative operators applied to the momentum equations are not intertwined
with the pressure equation, and thus we can apply the same operators to
the magnetic field equations to close the estimates on $\u$ and $\H$.
On the other hand, the proof presented in \cite{JJL4} fully exploits the structure
of the ideal MHD equations and cannot be directly extended to
the full compressible MHD equations studied in the current paper
where the heat conductivity is positive.
  \end{rem}

 We point out that the low Mach number limit is an interesting topic
 in fluid dynamics and applied mathematics. Below we briefly review
some related results on the Euler, Navier-Stokes and MHD equations.
In \cite{S86}, Schochet obtained the convergence of the non-isentropic
compressible Euler equations to the incompressible non-isentropic
Euler equations in a bounded domain for local smooth solutions and well-prepared initial data.
As mentioned above, in \cite{MS01} M\'{e}tivier and Schochet proved
rigorously the incompressible limit of the compressible non-isentropic
Euler equations in the whole space  with general initial data, see also \cite{A05, A06, LST}
for further extensions. In \cite{MS03} M\'{e}tivier and Schochet
showed the incompressible limit of the one-dimensional non-isentropic Euler equations
 in a periodic domain with general data. For compressible heat-conducting flows,
 Hagstrom and Lorenz established in \cite{HL02} the low Mach number limit
 under the assumption that the variation of the density and temperature is small.
 In the case of without heat conductivity, Kim and Lee \cite{KL05} investigated the incompressible limit
to the non-isentropic Navier-Stokes equations in a periodic domain with well-prepared data, while
Jiang and Ou \cite{JO} investigated the incompressible limit in a three-dimensional
bounded domain, also for well-prepared data. The justification of the low Mach number limit for
the non-isentropic Euler or Navier-Stokes equations with general initial
data in bounded domains or multi-dimensional periodic domains is still a challenging open problem.
We refer the interested reader to \cite{BDGL} on formal computations for viscous
polytropic gases, and to \cite{MS03,BDG} for the study on the
acoustic waves of the non-isentropic Euler equations in periodic domains.
Compared with the non-isentropic case, the description of
the propagation of oscillations in the isentropic case is simpler
and there are many articles on this topic (isentropic flows) in the literature, see, for example,
Ukai \cite{U86}, Asano \cite{As87}, Desjardins and Grenier \cite{DG99}, and Masmoudi \cite{M01} in
the whole space case; Isozaki \cite{I87,I89} in the case of exterior domains;
Iguchi \cite{Ig97} in the half space case; Schochet \cite{S94} and
Gallagher \cite{Ga01} in the case of periodic domains; and Lions and Masmoudi
\cite{LM98}, and Desjardins, et al. \cite{DGLM} in the case of bounded domains.

For the compressible isentropic MHD equations, the justification of
the low Mach number limit has been established in several aspects. In \cite{KM}
Klainerman and Majda studied the low Mach number limit to the
compressible isentropic MHD equations in the spatially periodic case with
well-prepared initial data. Li \cite{Li} considered the inviscid, incompressible limit of the viscous isentropic
compressible MHD  equations, also for well-prepared initial data,   by applying the
convergence-stability principle.
Recently, the low Mach number limit to the compressible
isentropic viscous (including both viscosity and magnetic diffusivity) MHD equations with
 general data was studied in \cite{HW3,JJL1,JJL2}. In
\cite{HW3} Hu and Wang obtained the convergence of weak solutions
to the compressible viscous MHD equations in bounded domains,
periodic domains and the whole space. In \cite{JJL1} Jiang, Ju and Li
employed the modulated energy method to verify the
limit of weak solutions of the compressible MHD equations in the
torus to the strong solution of the incompressible viscous or
partially viscous MHD equations (zero shear viscosity but with magnetic diffusion),
while in \cite{JJL2} the convergence of weak
solutions of the viscous compressible MHD equations to the strong
solution of the ideal incompressible MHD equations in the whole
space $\mathbb{R}^d (d=2,3)$ was established by using the dispersion property of the wave equation, as both
shear viscosity and magnetic diffusion coefficients go to zero.
For the full compressible MHD equations, the incompressible limit
 in the framework of the so-called variational solutions was established in \cite{K10,KT,NRT}.
Fan, Gao and Guo \cite{FGG} studied the low Mach number limit of the non-isentropic MHD equations
with zero thermal conductivity under the assumption
that the initial data are uniformly bounded with respect to the Mach number in $H^3(\mathbb{R}^3)$  and are well-prepared
in $H^1(\mathbb{R}^3)$. Recently, the low Mach number limit for the
full compressible MHD equations with small
entropy or temperature variation was justified rigorously in \cite{JJL3}.

 It should be remarked that only the case of the well-prepared initial
data has been treated in \cite{FGG, JJL3, Li} for the low Mach number limit under the framework of smooth
solutions with no oscillations. The uniform estimates presented in \cite{FGG, JJL3, Li} are mainly based on the standard theory of symmetrizable
  hyperbolic-parabolic equations.
Here we consider the low Mach number limit to the system \eqref{hnaaa}--\eqref{hnadd}
with large temperature variations and general (ill-prepared) initial data.
In this case two times scales appear     and the  theory of symmetrizable hyperbolic-parabolic equations
cannot be applied directly to obtain the uniform estimates. Moreover, we must consider the  acoustic wave caused by the
general initial data.
 Comparing the arguments on the magnetic field in \cite{FGG, JJL3,Li} with those of the present paper,
  we have to pay more attention to deal with the coupled terms between acoustic part of the fluid velocity and magnetic field.
 To overcome these additional difficulties, the main idea used here is to bound the acoustic part of the fluid velocity in terms of the norms
of $(\epsilon \partial_t) (p^\epsilon,\u^\epsilon, \H^\epsilon)$ and
$(\epsilon p^\epsilon,\epsilon \u^\epsilon,\epsilon \H^\epsilon, \theta^\epsilon )$
 through the density and the momentum equations and to exploit the structure of the system.
 In addition, the commutator estimates (see Lemmas \ref{Lb} and \ref{Lc} below)
  are used very carefully in the process of uniform estimates.

Besides the references mentioned above, we refer the interested reader
to the monograph \cite{FN} and the survey papers
 \cite{Da05,M07,S07} for more related results on the low Mach number limit for fluid models.

We also mention that there are a lot of articles in the literature on the other topics related
to the compressible MHD equations due to theirs physical importance, complexity, rich
phenomena, and mathematical challenges,
see, for example, \cite{BT02,CW02,CW03,FS,Go,LY11,DF06,FJN,HT,LL,HW1,HW2,ZJX} and the
references cited therein.

This paper is arranged as follows. In Section 2, we describe some notations, recall basic facts and
 present commutator estimates. In Section 3 we first establish a priori estimates
 on $(\H^\epsilon,\theta^\epsilon)$,
 $(\epsilon p^\epsilon, \epsilon \u^\epsilon, \epsilon \H^\epsilon, \theta^\epsilon)$ and
on $(p^\epsilon, \u^\epsilon)$. Then, with the help of these estimates we establish
the uniform boundedness of the solutions and prove the existence part of Theorem \ref{main}.
 Finally, in Section 4 we study the local energy decay for the acoustic wave equations
 and prove the convergence part of Theorem \ref{main}.

\section{Preliminary}

 In this section, we  give some notations and recall some basic facts
which will be frequently used throughout the paper. We also present
some commutators estimates introduced in \cite{KM,LST} and state the results
on local solutions to the Cauchy problem \eqref{hnaaa}--\eqref{hnadd}, \eqref{hnas}.

 We denote by $\langle\cdot,\cdot\rangle$ the standard inner
product in $L^2(\R^3)$ with norm $\langle f,f\rangle=\|f\|^2_{L^2}$ and by
 $H^k$ the standard Sobolev space $W^{k,2}$ with norm $\|\cdot\|_{H^k}$.
 The notation $\|(A_1,\dots, A_k)\|_{L^2}$ means
the summation of $\|A_i\|_{L^2},i=1,\cdots,k$, and it also applies to other norms.
For a multi-index $\alpha = (\alpha_1,\alpha_2,\alpha_3)$, we denote
$\partial^\alpha =\partial^{\alpha_1}_{x_1}\partial^{\alpha_2}_{x_2}\partial^{\alpha_3}_{x_3}$
and $|\alpha|=|\alpha_1|+|\alpha_2|+|\alpha_3|$. We will omit the spatial
domain $\R^3$ in integrals for convenience. We use $l_i>0$ ($i\in\mathbb{N}$) to denote given
constants. We also use the symbol $K$ or $C_0$ to denote generic positive constants,
and $C(\cdot)$ to denote a smooth function which may vary from line to line.

  For a scalar function $f$, vector functions $\mathbf{a}$ and $\mathbf{b}$, we have
  the following basic vector identities:
\begin{align}
 & \dv(f \mathbf{a}) = f  \dv \mathbf{a} +\nabla f\cdot   \mathbf{a}, \label{vbd}\\
& \cu(f \mathbf{a}) =f\cdot \cu \mathbf{a} -\nabla f\times   \mathbf{a}, \label{vbc}\\
&\dv (\mathbf{a}\times \mathbf{b}) =\mathbf{b}\cdot\cu \mathbf{a} -\mathbf{a}\cdot\cu\mathbf{b},\label{va}\\
&\cu(\mathbf{a}\times\mathbf{b}) =(\mathbf{b}\cdot\nabla)\mathbf{a} - (\mathbf{a}\cdot\nabla)\mathbf{b}
  + \mathbf{a} (\dv \mathbf{b})  -   \mathbf{b} (\dv \mathbf{a}),\label{vb}\\
&\nabla (\mathbf{a}\cdot \mathbf{b})
=(\mathbf{a}\cdot \nabla )\mathbf{b}+(\mathbf{b}\cdot \nabla)\mathbf{a}
+\mathbf{a}\times (\cu \mathbf{b})+\mathbf{b}\times (\cu \mathbf{a}).  \label{vv}
 \end{align}

%
Below we recall some results on commutator estimates.
\begin{lem}[\cite{KM}]\label{Lb}
Let $k>5/2$ be an integer and $\alpha=(\alpha_1,\alpha_2,\alpha_3)$ be a multi-index such
that $|\alpha|=k$. Then, for any $\sigma\geq 0$, there exists a positive constant $C_0$,
such that for all $f,g\in H^{k+\sigma}(\R^3)$,
\begin{align}\label{comc}
    \|[f,\partial^\alpha] g\|_{H^{\sigma}}\leq\, & C_0(\|f\|_{W^{1,\infty}}\| g\|_{H^{\sigma+k-1}}
    +\| f\|_{H^{\sigma+k}}\|g\|_{L^{\infty}}).
\end{align}
\end{lem}

\begin{lem}[\cite{LST}]\label{Lc}
Let $s\geq 4$ be an integer. Then there exists a positive constant $C_0$, such that
 for all $\epsilon\in (0,1]$, $T>0$ and multi-index $\beta=(\beta_1,\beta_2,\beta_3)$
 satisfying $0\leq |\beta|\leq s-1$, and any $f,g\in C^\infty([0,T],H^{s}(\R^3))$, it holds that
\begin{align}\label{comd}
    \|[f,\partial^\beta(\epsilon\partial_t)] g\|_{L^2}\leq\, & \epsilon C_0
(\|f\|_{H^{s-1}}\|\partial_t g\|_{H^{s-2}} +\|\partial_t f\|_{H^{s-1}}\| g\|_{H^{s-1}}).
\end{align}
\end{lem}

Since the system \eqref{haa}--\eqref{hac}, \eqref{hagg},  \eqref{hpg}
 is hyperbolic-parabolic, thus the classical result of Vol'pert
and Hudiaev \cite{VK} implies that
\begin{prop}\label{hPa}
Let $s\geq 4$  be an integer. Assume that the initial data $(\rho_0,\u_0 ,\H_0,\linebreak \theta_0)$ satisfy
\begin{align*}
   \|(\rho_0 -\underline\rho, \u_0,\H_0,\theta_0 -\underline\theta)\|_{H^s}\leq C_0
\end{align*}
for some positive constants $\underline\rho$, $\underline\theta$ and $C_0$.
Then there exists a $\tilde T>0$, such that
 the system \eqref{haa}--\eqref{hac}, \eqref{hagg},  \eqref{hpg}
 with these initial data has a unique classical solution $(\rho,\u,\H,\theta)$ enjoying
$\rho-\underline\rho\in C([0,\tilde T],H^s(\R^3))$,
$(\u,\H,\theta-\underline \theta)\in C([0,\tilde T], \linebreak H^s(\R^3))\cap L^2(0,\tilde{T};H^{s+1}(\R^3))$, and
\begin{align*}
& \sup_{0\leq t\leq\tilde  T}\|(\rho-\underline\rho,\u,\H,\theta-\underline\theta)\|^2_{H^s}\nonumber\\
 & \quad \quad+ \int^{\tilde T}_0
 \big\{\mu \|\mathbb{D}(\u)\|^2_{H^{s}}+\lambda\|\dv \u\|^2_{H^{s}}   
+\nu \|\nabla \H\|^2_{H^{s}}+\kappa \|\nabla \theta\|^2_{H^{s}}
\big\}(\tau){\rm d}\tau\leq 4C_0^2.
\end{align*}
\end{prop}

It follows from Proposition \ref{hPa}, and the transforms \eqref{transa} and \eqref{transb}
that there exists a $T_\epsilon>0$, depending on $\epsilon$ and $L_0$, such that
for each fixed $\epsilon$ and any initial data \eqref{hnas} satisfying \eqref{initial},
the Cauchy problem  \eqref{hnaaa}--\eqref{hnadd}, \eqref{hnas} has a unique
solution $(p^\epsilon,\u^\epsilon,\H^\epsilon,\theta^\epsilon-\bar \theta)$ satisfying
$(p^\epsilon,\u^\epsilon,\H^\epsilon,\theta^\epsilon-\bar\theta)\in C([0,T_\epsilon),H^{s}(\R^3))$
and $(\u^\epsilon,\H^\epsilon, \theta^\epsilon-\bar\theta)\in L^2(0,T_\epsilon;H^{s+1}(\R^3))$.
Moreover, let $T_\epsilon^*$ be the maximal time of existence of such a smooth solution,
then if $T_\epsilon^*$ is finite, one has
\begin{align*}
{\underset{t\rightarrow T_\epsilon^*} {\lim \sup}}\,
  \left\{ \|(p^\epsilon,\u^\epsilon,\H^\epsilon,\theta^\epsilon-\bar \theta)(t)\|_{W^{1,\infty}}
  +\|(\u^\epsilon,\H^\epsilon,\theta^\epsilon-\bar \theta)(t)\|_{W^{2,\infty}}\right\}=\infty.
\end{align*}

Therefore, we shall see by the same argument as in \cite{MS01}
that the existence part of Theorem \ref{main}
is a consequence of the above assertion and the following key a priori estimates which
will be shown in the next section.
\begin{prop}\label{hPb}
 For any given integer $s\geq 4$ and fixed $\epsilon>0$, let $(p^\epsilon,\u^\epsilon,\H^\epsilon, \theta^\epsilon)$
be the classical solution to the Cauchy problem \eqref{hnaaa}--\eqref{hnadd}, \eqref{hnas}. Denote
\begin{align}
 & \mathcal{O}(T):=\|(p^\epsilon,\u^\epsilon,\H^\epsilon,\theta^\epsilon-\bar \theta)(T)\|_{s,\epsilon}, \nonumber\\
  & \mathcal{O}_0: =\|(p^\epsilon_{\rm in},\u^\epsilon_{\rm in},\H^\epsilon_{\rm in})\|_{H^{s}}+  \|(\epsilon p^\epsilon_{\rm in},\epsilon\u^\epsilon_{\rm in},
\epsilon\H^\epsilon_{\rm in}, \theta^\epsilon_{\rm in}-\bar \theta)\|_{H_\epsilon^{s+2}}.\nonumber
\end{align}
Then there exist positive constants $\hat{T}_0$ and $\epsilon_0<1$, and an increasing
positive function $C(\cdot)$, such that for all $T\in [0,\hat T_0]$ and $\epsilon \in (0,\epsilon_0]$,
 \begin{align*}
 \mathcal{O}(T) \leq C(\mathcal{O}_0)\exp\big\{(\sqrt{T}+\epsilon)C(\mathcal{O}(T))\big\}.
 \end{align*}
 \end{prop}

\section{Uniform estimates}

In this  section we shall establish the uniform bounds of the solutions to the Cauchy problem
 \eqref{hnaaa}--\eqref{hnadd}, \eqref{hnas} stated in Proposition \ref{hPb} by modifying the approaches
developed in \cite{MS01,A06,LST} and making careful use of the special structure of the system
\eqref{hnaaa}--\eqref{hnadd}. In the rest of this section, we will drop the superscripts $\epsilon$
of the variables in the Cauchy problem and denote
\begin{equation*}  
\Psi(\u)=2\bar \mu \mathbb{D}(\u)+\bar \lambda\dv\u \;\mathbf{I}_3.
\end{equation*}
Recall that it has been assumed that $\mu^\epsilon \equiv \bar\mu>0$, $\nu^\epsilon \equiv \bar\nu>0$,
$\kappa^\epsilon \equiv \bar\kappa>0$, and $\lambda^\epsilon \equiv \bar\lambda$ independent of $\epsilon$.

\subsection{$H^s$-estimates on $(\H,\theta)$ and $ (\epsilon p,\epsilon\u )$ }

To prove Proposition \ref{hPb}, we first give some
estimates derived directly from the system \eqref{hnaaa}--\eqref{hnadd}. Denoting
\begin{align}
  \mathcal{Q}(t):& =\|(  p, \u, \H, \theta-\bar \theta)(t) \|_{H^{s}}
  +\|(\epsilon p,\epsilon\u,\epsilon\H, \theta-\bar \theta)(t) \|_{H_\epsilon^{s+2}}, \nonumber\\
 \mathcal{S}(t):& =\| (\nabla\u, \nabla p,\nabla\H)(t)\|_{H^s}
 +\|\nabla (\epsilon\u,\epsilon\H,\theta)(t)\|_{H_\epsilon^{s+2}}, \nonumber
\end{align}
one has
\begin{lem}\label{HES}
 Let $s\geq 4$  be an integer and $(p,\u,\H,\theta)$ be a solution to the problem \eqref{hnaaa}--\eqref{hnadd},
 \eqref{hnas} on $[0,T_1]$.
There exists an increasing function $C(\cdot)$ such that, for any $\epsilon \in (0,1]$ and
$t\in [0,T]$, $T=\min\{T_1,1\}$, it holds that
 \begin{align}
&\sup_{\tau\in[0,t]}\big\{\|(\H,\theta)(\tau)\|_{H^s}+\|\epsilon\H(\tau)\|_{H_\epsilon^{s+2}}\big\}\nonumber\\
&\quad +\bigg\{\int^t_0\big(\|\nabla(\H,\theta)\|_{H^{s}}^2+ \|\nabla(\epsilon\H)\|^2_{H_\epsilon^{s+2}}\big)(\tau){\rm d}\tau\bigg\}^{1/2}
 \leq C(\mathcal{O}_0)\exp\big\{\sqrt{T}C(\mathcal{O}(T))\big\}. \nonumber 
\end{align}
\end{lem}
\begin{proof}
For any multi-index $\alpha$ satisfying $0\leq |\alpha|\leq s$,
let $\H_\alpha=\partial^\alpha\H$. Then
\begin{align*}
\partial_t\H_\alpha+(\u\cdot\nabla)\H_\alpha-\bar{\nu}\Delta\H_\alpha=-[\partial^\alpha,\u]\cdot\nabla\H
-\partial^\alpha(\H\dv\u)+\partial^\alpha((\H\cdot\nabla)\u).
\end{align*}
Taking inner product of the above equations with $\H_\alpha$ and integrating by parts, we have
\begin{align*}
\frac{1}{2}\frac{\rm d}{{\rm d}t}\|\H_\alpha\|_{L^2}^2+\bar{\nu}\|\nabla\H_\alpha\|_{L^2}^2=&-\langle(\u\cdot\nabla)
\H_\alpha,\H_\alpha\rangle-\langle[\partial^\alpha,\u]\cdot\nabla\H,\H_\alpha\rangle\nonumber\\
 &-\langle\partial^\alpha(\H\dv\u),\H_\alpha\rangle+
\langle\partial^\alpha((\H\cdot\nabla)\u),\H_\alpha\rangle.
\end{align*}
By integration by parts we obtain
\begin{align*}
-\langle(\u\cdot\nabla) \H_\alpha,\H_\alpha\rangle=\frac{1}{2}\int \dv
\u|\H_\alpha|^2 dx\leq C(\mathcal{Q})\| \H_\alpha\|^2_{L^2}.
\end{align*}
It follows from Lemma \ref{Lb} that
$$  \|[\partial^\alpha,\u]\cdot\nabla\H\|_{L^2}\leq
C_0(\|\u\|_{W^{1,\infty}}\|\nabla\H\|_{H^{s-1}}+\|\u\|_{H^{s}}\|\nabla\H\|_{L^\infty})
\leq C(\mathcal{Q}). $$
By Sobolev's inequality, one gets
$$ -\langle\partial^\alpha(\H\dv\u),\H_\alpha\rangle +
\langle\partial^\alpha(\H\cdot\nabla\u),\H_\alpha\rangle
\leq C_0\|\H\|_{H^s}^2\|\u\|_{H^{s+1}}\leq C(\mathcal{Q})\mathcal{S}. $$
Thus, we conclude that
\begin{align*}
&  \sup_{\tau\in[0,t]}\|\H(\tau)\|_{H^s}+\bar{\nu}\bigg\{\int^t_0\|\nabla\H(\tau)\|^2_{H^s}{\rm d}\tau\bigg\}^{1/2}\\
&\quad \leq
C(\mathcal{O}_0)+C(\mathcal{O}(t)t+C(\mathcal{O}(t))\int_0^t\mathcal{S}(\tau){\rm d}\tau  \\
&\quad\leq C(\mathcal{O}_0)+C(\mathcal{O}(t)t+C(\mathcal{O}(t))\sqrt{t}  \\
&\quad\leq C(\mathcal{O}_0)+C(\mathcal{O}(T))\sqrt{T}  \\
&\quad\leq C(\mathcal{O}_0)\exp\big\{\sqrt{T}C(\mathcal{O}(T))\big\}.
\end{align*}

Now denote $\hat{\H}=\epsilon \H$ and $\hat{\H}_\alpha=\partial^\alpha(\epsilon\H)$
for $0\leq |\alpha|\leq  s+1$. Then, $\hat{\H}_\alpha$ satisfies
\begin{align}\label{hhs}
& \partial_t\hat{\H}_\alpha+(\u\cdot\nabla)\hat{\H}_\alpha-\bar{\nu}\Delta\hat{\H}_\alpha\nonumber\\
& \qquad =-\epsilon[\partial^\alpha,\u]\cdot\nabla\H
-\epsilon\partial^\alpha(\H\dv\u)+\epsilon\partial^\alpha((\H\cdot\nabla)\u).
\end{align}
Applying Lemma \ref{Lb} again implies that
$$ \|-\epsilon[\partial^\alpha,\u]\cdot\nabla\H\|_{L^2}\leq
C_0(\|\u\|_{W^{1,\infty}}\|\nabla\hat{\H}\|_{H^s}
+\|\epsilon\u\|_{H^{s+1}}\|\nabla\H\|_{L^\infty}) \leq C(\mathcal{Q}),  $$
while an integration by parts and Sobolev's inequality lead to
\begin{align*}
-\langle\epsilon\partial^\alpha(\H\dv\u),\hat{\H}_\alpha\rangle
+\langle\epsilon\partial^\alpha((\H\cdot\nabla)\u),\hat{\H}_\alpha\rangle
& \leq \frac{\bar{\nu}}{2}\|\nabla\hat{\H}_\alpha\|_{L^2}^2
+C_0\|\H\|_{H^s}^2\|\epsilon\u\|_{H^{s+1}}^2\\
& \leq \frac{\bar{\nu}}{2}\|\nabla\hat{\H}_\alpha\|_{L^2}^2+C(\mathcal{Q}).
\end{align*}
Hence, after summing over $\alpha$ for $0\leq |\alpha|\leq  s+1$, we obtain that
\begin{align*}
\sup_{\tau\in[0,t]}\|\epsilon\H(\tau)\|_{H^{s+1}}+\Big\{\int^t_0\bar{\nu}\|\nabla(\epsilon\H)(\tau)\|^2_{H^{s+1}}{\rm d}\tau\bigg\}^{1/2}&\leq
C(\mathcal{O}_0)\exp\big\{\sqrt{T}C(\mathcal{O}(T))\big\}.
\end{align*}
Similarly,  we can obtain
\begin{align*}
\sup_{\tau\in[0,t]}\epsilon^2\|\H(\tau)\|_{H^{s+2}}+\epsilon^2\bigg \{\int^t_0\bar{\nu}\|\nabla\H(\tau)\|^2_{H^{s+2}}{\rm d}\tau\bigg\}^{1/2}&\leq
C(\mathcal{O}_0)\exp\big\{\sqrt{T}C(\mathcal{O}(T))\big\}.
\end{align*}

Next, we estimate $\theta$. Using Sobolev's inequality, one finds that
\begin{align*}
& \|\partial^s(\epsilon^2e^{-\epsilon p}[\bar{\nu}|\cu\H|^2
+\Psi(\u):\nabla\u])\|_{L^2}\nonumber\\
& \qquad\qquad\qquad \leq C_0\|\epsilon
p\|_{H^s}(\|\epsilon\nabla\H\|^2_{H^{s}}+\|\epsilon \nabla\u\|^2_{H^s})\leq
C(\mathcal{Q}).
\end{align*}
Employing arguments similar to those used for $\H$, we can obtain
\begin{align*}
\sup_{\tau\in[0,t]}\|(\theta-\bar \theta)(\tau)\|_{H^{s}}+\bigg \{\int^t_0\bar{\kappa}\|\nabla\theta(\tau)\|^2_{H^{s}}{\rm d}\tau\bigg\}^{1/2}&\leq
C(\mathcal{O}_0)\exp\big\{\sqrt{T}C(\mathcal{O}(T))\big\}.
\end{align*}
Thus, the lemma is proved.
\end{proof}

\begin{lem}\label{EP1}
Let $s\geq 4$  be an integer and $(p,\u,\H,\theta)$ be a solution to the problem \eqref{hnaaa}--\eqref{hnadd},
 \eqref{hnas} on $[0,T_1]$. Then there exists an increasing function $C(\cdot)$ such that,
for any $\epsilon \in (0,1]$ and $t\in [0,T], T=\min\{T_1,1\}$, it holds that
\begin{align}
&\sup_{\tau\in
[0,t]}\|(\epsilon p,\epsilon\u)(\tau)\|_{H^s}+\bigg \{\int^t_0\bar{\mu}
\|\nabla (\epsilon\u )(\tau)\|^2_{H^{s}}{\rm d}\tau\bigg\}^{1/2} \leq
C(\mathcal{O}_0)\exp\big\{\sqrt{T}C(\mathcal{O}(T))\big\}.\nonumber 
\end{align}
\end{lem}
\begin{proof}
Let $\check{p}=\epsilon p$, and $\check{p}_\alpha=\partial^\alpha(\epsilon p)$
for any multi-index $\alpha$ satisfying  $0\leq|\alpha|\leq s$. Then
\begin{align}
\partial_t \check{p}_\alpha+(\u\cdot\nabla)\check{p}_\alpha=\,&-[\partial^\alpha,\u]
\cdot(\nabla\check{p})-\partial^\alpha[\dv(2\u-\kappa a(\epsilon p)b(\theta)\nabla \theta)]\nonumber\\
 & +\partial^\alpha \{a(\epsilon p) [\nu|\cu(\epsilon\H)|^2
+\Psi(\epsilon\u):\nabla(\epsilon\u)]\}\nonumber\\
&+\kappa \partial^\alpha\{ a(\epsilon p)b(\theta)\nabla(\epsilon
p)\cdot\nabla\theta\}\nonumber\\
:=\,&h_1+h_2+h_3+h_4,  \label{ee0}
\end{align}
where, for simplicity of presentation, we have set
$$ a(\epsilon p): = e^{-\epsilon p}, \quad b(\theta): =e^{\theta}.$$

It is easy to see that the energy estimate for \eqref{ee0} gives
\begin{align}\label{ee1}
\frac{1}{2}\frac{\rm d}{{\rm d}t}\|\check{p}_\alpha\|_{L^2}^2=-\langle(\u\cdot\nabla)\check{p}_\alpha,
\check{p}_\alpha\rangle+\langle h_1+h_2+h_3+h_4, \check{p}_\alpha\rangle ,
\end{align}
where we have to estimate each term on the right-hand side of \eqref{ee1}.
First, an integration by parts yields
\begin{align*}
-\langle (\u\cdot\nabla)\check{p}_\alpha,
\check{p}_\alpha\rangle=\frac{1}{2}\int \dv \u|\check{p}_\alpha|^2
{\rm d}x\leq C(\mathcal{Q})\| \check{p}_\alpha\|^2_{L^2},
\end{align*}
while the commutator inequality \eqref{comc} leads to
\begin{align*}
\|h_1\|\leq C_0(\|\u\|_{W^{1,\infty}}\|\nabla\check{p}\|_{H^{s-1}}+
\|\u\|_{H^s}\|\nabla\check{p}\|_{L^\infty})\leq C(\mathcal{Q}).
\end{align*}
Consequently,
\begin{align*}
\langle h_1,\check{p}_\alpha\rangle\leq
\|\check{p}_\alpha\|_{L^2}\|h_1\|_{L^2}\leq C(\mathcal{Q}).
\end{align*}

From Sobolev's inequality one gets
$$
\|h_2\|\leq C_0\|\u\|_{H^{s+1}}+\|\theta\|_{H^{s}}\|\epsilon
p\|_{H^s}\|\theta\|_{H^{s+2}} \leq C(\mathcal{S})(1+ C(\mathcal{Q})), $$
whence,
\begin{align*}
\langle h_2, \check{p}_\alpha\rangle \leq\, &
C(\mathcal{Q})C(\mathcal{S}).
\end{align*}

Similarly, one can prove that
\begin{align*}
\langle h_3+h_4, \check{p}_\alpha\rangle \leq\, & C(\mathcal{Q}).
\end{align*}
Hence, we conclude that
\begin{align*}
\sup_{\tau\in
[0,t]}\|\epsilon p (\tau)\|_{H^s}\leq
C(\mathcal{O}_0)\exp\big\{\sqrt{T}C(\mathcal{O}(T))\big\}.
\end{align*}
In a similar way, we can estimate $\u$. Thus the proof of the lemma is completed.
\end{proof}

Next, we control the norm $\|(\u,p)\|_{H^s}$. The idea is to bound the
norms of $(\dv \u,$ $ \nabla p)$ in terms of the suitable norm of $(\epsilon\u, \epsilon p,
\epsilon\H,\theta)$ and $\epsilon(\partial_t\u,\partial_t p)$ by making use of
the structure of the system. To this end, we first estimate
$\|(\epsilon\u,\epsilon p,\theta)\|_{H^{s+1}}$.

\subsection{$H^{s+1}$-estimates on $(\epsilon\u,\epsilon p,\epsilon \H,\theta)$ }

Following \cite{A06}, we set
\begin{align*}
   (\hat{p}, \hat \u,\hat \H,\hat \theta):= (\epsilon p-\theta, \epsilon \u, \epsilon\H, \theta-\bar \theta).
\end{align*}
A straightforward calculation implies that $(\hat p, \hat \u,\hat\H,\hat \theta)$
solves the following system:
\begin{align}
   &\partial_t \hat{p}  +(\u \cdot\nabla)\hat{p}
   +\frac{1}{\epsilon}\dv  \hat \u =0, \label{slowf} \\
&b(-\theta)[\partial_t\hat \u +(\u\cdot\nabla)\hat\u ]+\frac{1}{\epsilon}
(\nabla \hat p+\nabla \hat \theta)\nonumber\\
 &  \qquad \qquad \qquad  =a( \epsilon p)[(\cu \H)\times \hat\H+\dv\Psi (\hat \u )], \label{slowg} \\
&\partial_t\hat \H +\u\cdot\nabla\hat
\H+\H\dv\hat\u-\H\cdot\nabla\hat\u- \bar\nu \Delta\hat \H = 0,\quad
\dv\hat \H =0,\label{slowh} \\
&\partial_t \hat \theta  +(\u\cdot\nabla)\hat \theta +
\frac{1}{\epsilon}\dv\hat \u  = \epsilon a(\epsilon p)[ \bar \nu\,
\cu\H :\cu\hat \H
+\epsilon a(\epsilon p)\Psi (\u ):\nabla\hat \u ]\nonumber\\
& \qquad \qquad \qquad \qquad \qquad\  \quad +\bar \kappa
a(\epsilon p)\dv (b(\theta)\nabla \hat  \theta ).
  \label{slowi}
\end{align}

We have
\begin{lem}\label{LLb}
 Let $s\geq 4$  be an integer and $(p,\u,\H,\theta)$ be a solution to the problem \eqref{hnaaa}--\eqref{hnadd},
 \eqref{hnas} on $[0,T_1]$. Then there exist a constant $l_1>0$ and an increasing function $C(\cdot)$ such that,
 for any $\epsilon \in (0,1]$ and $t\in [0,T]$, $T=\min\{T_1,1\}$, it holds that
\begin{align}\label{slowao}
 & \sup_{\tau\in [0,t]}\|(\epsilon q,\epsilon\u, \theta-\bar \theta)(\tau)\|_{H^{s+1}}
+l_1\bigg \{\int^t_0\|\nabla (\epsilon\u,\theta)(\tau)\|^2_{H^{s+1}}{\rm d}\tau\bigg\}^{1/2}\nonumber\\
 &\quad  \leq  C(\mathcal{O}_0)\exp\big\{\sqrt{T}C(\mathcal{O}(T))\big\}.
 \end{align}

\end{lem}

\begin{proof}
Let $\alpha$ be a multi-index such that $0\leq|\alpha|\leq s+1$. Set
\begin{align*}
  (\hat{p}_\alpha,\hat{\u}_\alpha,\hat{H}_\alpha,\hat {\theta}_\alpha):=
  \left(\partial^\alpha(\epsilon p-\theta), \partial^\alpha
   (\epsilon \u),\partial^\alpha(\epsilon\H), \partial^\alpha
     (\theta-\bar\theta)\right).
\end{align*}
Then, $\hat{\H}_\alpha$ satisfies \eqref{hhs} and
$(\hat{p}_\alpha,\hat{\u}_\alpha,\hat\theta_\alpha)$ solves
\begin{align}
   &\partial_t \hat{p}_\alpha  +(\u   \cdot\nabla)\hat{p}_\alpha
   +\frac{1}{\epsilon}\dv  \hat{\u}_\alpha =g_1, \label{slowap} \\
&b(-\theta )[\partial_t\hat{\u}_\alpha +(\u  \cdot\nabla)\hat{\u}_\alpha ]
+\frac{1}{\epsilon}(\nabla \hat{p}_\alpha+\nabla \hat{\theta}_\alpha)\nonumber\\
 & \quad \quad \qquad \quad = a( \epsilon p )(\cu \H )\times \hat{\H}_\alpha
 +a( \epsilon p )\dv\Psi (\hat{\u}_\alpha )+g_2, \label{slowaq} \\
&\partial_t \hat\theta_\alpha  +(\u  \cdot\nabla)\hat\theta_\alpha +
\frac{1}{\epsilon}\dv\hat{\u}_\alpha
   =\epsilon a(\epsilon p )[
\bar \nu\, \cu\H  :\cu\hat\H_\alpha
+\Psi (\u  ):\nabla\hat\u_\alpha] \nonumber\\
& \quad \quad \qquad \qquad \qquad \quad \qquad \qquad +\bar \kappa  a(\epsilon p )\dv (b(\theta
)\nabla  \hat \theta_\alpha )+g_3,
  \label{slowas}
\end{align}
with initial data
\begin{align}
(\hat{p}_\alpha ,\hat{\u}_\alpha ,\hat{\H}_\alpha , \hat\theta_\alpha
)|_{t=0} :=\big( & \partial^\alpha(\epsilon p_{\rm in}(x)
- \theta_{\rm in}(x)),\partial^\alpha(\epsilon\u_{\rm in}(x)),\nonumber\\
& \partial^\alpha(\H_{\rm in}(x)), \partial^\alpha(\theta_{\rm in}(x)-\bar
\theta)\big),\label{slowat}
\end{align}
where
\begin{align}
g_1:=&  -[\partial^\alpha, \u]\cdot \nabla(\epsilon p-\theta),\nonumber\\
g_2:=&  -[\partial^\alpha, b(-\theta)]\partial_t(\epsilon \u)
       -[\partial^\alpha, b(-\theta)\u]\cdot \nabla (\epsilon \u)\nonumber\\
    &  +[\partial^\alpha, a(\epsilon p)\cu (\epsilon\H)]\times \H
    +[\partial^\alpha, a(\epsilon p)]\dv \Psi(\epsilon \u), \nonumber\\
g_3:=& -[\partial^\alpha, \u]\cdot \nabla \theta +\bar \nu\,[\partial^\alpha, a(\epsilon p )
 \cu(\epsilon\H)]  :\cu(\epsilon\H)\nonumber\\
& +\epsilon[\partial^\alpha, a(\epsilon p)\Psi (\u
)]:\nabla(\epsilon\u) +\bar \kappa \partial^\alpha\big ( a(\epsilon p )\dv (b(\theta )\nabla
\theta) \big)\nonumber\\
& -\bar \kappa a(\epsilon p) \dv  (b(\theta)\nabla \hat  \theta_\alpha).
 \nonumber 
\end{align}

It follows from Proposition \ref{hPa} and the positivity of $a(\cdot)$ and
$b(\cdot)$ that $a(\cdot)$ and $b(\cdot)$ are bounded
away from $0$ uniformly with respect to $\epsilon$, i.e.
\begin{align}
a(\epsilon p)\geq \underline{a}>0,\;\;\;b(-\theta)\geq
\underline{b}>0.\label{ab1}
\end{align}

The standard  $L^2$-energy estimates for
\eqref{slowap}, \eqref{slowaq} and \eqref{slowas} yield that
\begin{align}
&\frac{1}{2}\frac{\rm d}{{\rm d}t}\big(\|\hat{p}_\alpha\|_{L^2}^2+\langle
b(-\theta)\hat{\u}_\alpha, \hat{\u}_\alpha\rangle+
\|\hat \theta_\alpha\|_{L^2}^2\big)\nonumber\\
&\quad \leq \frac{1}{2}\langle
b_t(\theta)\hat\u_\alpha,\hat\u_\alpha\rangle-\langle (\u
\cdot\nabla)\hat{p}_\alpha,\hat{p}_\alpha\rangle-\langle
b(-\theta)(\u \cdot\nabla)\hat{\u}_\alpha,
\hat{\u}_\alpha\rangle\nonumber\\
&\quad\quad -\langle (\u
\cdot\nabla)\hat \theta_\alpha,\hat \theta_\alpha \rangle+\langle a( \epsilon p )(\cu \H )\times \hat{\H}_\alpha,
\hat\u_\alpha\rangle+\langle a( \epsilon p )\dv\Psi (\hat{\u}_\alpha ),
\hat\u_\alpha\rangle\nonumber\\
&\quad\quad +\langle\epsilon a(\epsilon p )[ \bar \nu\, \cu\H  :\cu\hat\H_\alpha
+\Psi (\u  ):\nabla\hat\u_\alpha],\theta_\alpha\rangle\nonumber\\
&\quad\quad +\langle \bar
\kappa  a(\epsilon p )\dv (b(\theta )\nabla   \theta_\alpha
),\hat\theta_\alpha\rangle+\langle g_1,\hat{p}_\alpha\rangle+\langle
g_2,\hat{u}_\alpha\rangle+\langle g_3,\hat\theta_\alpha\rangle . \label{3.21}
\end{align}
It follows from equation \eqref{hnadd}, and the definitions of $\mathcal{Q}$
and $\mathcal{S}$ that
\begin{align*}
\|b_t(\theta)\|_{L^\infty}\leq
\|b(\theta)\|_{H^s}\|\theta_t\|_{H^s}\leq
C(\mathcal{Q})(1+\mathcal{S}).
\end{align*}
Therefore,
\begin{align*}
\frac{1}{2}\langle b_t(\theta)\hat\u_\alpha,\hat\u_\alpha\rangle
\leq C(\mathcal{Q})(1+\mathcal{S}).
\end{align*}

On the other hand, it is easy to see that
\begin{align*}
-\langle (\u
\cdot\nabla)\hat{p}_\alpha,\hat{p}_\alpha\rangle-\langle
b(-\theta)(\u \cdot\nabla)\hat{\u}_\alpha,
\hat{\u}_\alpha\rangle-\langle (\u
\cdot\nabla)\hat\theta_\alpha,\hat\theta_\alpha \rangle \leq C(\mathcal{Q})
\end{align*}
and \begin{align*} \langle a( \epsilon p )(\cu \H )\times
\hat{\H}_\alpha, \hat\u_\alpha\rangle\leq C(\mathcal{Q}).
\end{align*}
By integration by parts we have
\begin{align}\label{slowt}
-\langle a(\epsilon p_0) \dv \Psi(\hat \u), \hat \u\rangle
= & \int \mu a(\epsilon p_0)\big(|\nabla\hat \u_\alpha|^2+(\mu+\lambda) |\dv \hat \u_\alpha|^2\big){\rm d}x\nonumber\\
& + \mu\langle(\nabla a(\epsilon p)\cdot\nabla)\hat\u_\alpha, \hat \u_\alpha\rangle\nonumber\\
& + (\mu+\lambda)\langle  \nabla a(\epsilon p)\dv \hat \u_\alpha,  \hat \u_\alpha\rangle\nonumber\\
:=\, &\, d_1+d_2+d_3.
\end{align}
Thanks to the assumption that $\bar \mu>0$ and  $2\bar \mu +3\bar
\lambda>0$, there exists a positive constant $\xi_1$, such that
\begin{align}\label{slowu}
  d_1& \geq  {\underline a}\xi\int|\nabla\hat\u_\alpha|^2{\rm d}x ,
\end{align}
while Cauchy-Schwarz's inequality implies
\begin{align}\label{slowv}
 | d_2|+|d_3|\leq C(\mathcal{Q})\mathcal{S}.
\end{align}

Similarly, we can obtain
\begin{align}\label{sloww}
   -\langle \bar \kappa  a(\epsilon p)\dv (b(\theta)\nabla  \hat\theta_\alpha ),\hat\theta_\alpha\rangle
\geq  \bar \kappa {\underline a}\,{\underline b}\|\nabla \hat
\theta_\alpha\|^2_{L^2}-C(\mathcal{Q})\mathcal{S}.
\end{align}
Easily, one has
\begin{align*}
|\langle\epsilon a(\epsilon p )[ \bar \nu\, \cu\H  :\cu\hat\H_\alpha
+\Psi (\u  ):\nabla\hat\u_\alpha],\hat\theta_\alpha\rangle|\leq
C(\mathcal{Q})(1+\mathcal{S}).
\end{align*}

It remains to estimate $\langle g_1,\hat{p}_\alpha\rangle$,
$\langle g_2,\hat{u}_\alpha\rangle$ and $\langle g_3,\hat\theta_\alpha\rangle$
in \eqref{3.21}. First, an application of H\"{o}lder's inequality gives
\begin{align*}
|\langle \hat{p}_\alpha,  g_1\rangle|\leq
C_0\|\hat{p}_\alpha\|_{L^2}\|g_1\|_{L^2},
\end{align*}
where $\|g_1\|_{L^2}$ can be bounded, by using  \eqref{comc}, as follows
\begin{align*}
\|g_1\|_{L^2}=\,&  \|[\partial^\alpha, \u]\cdot \nabla(\epsilon p-\theta)\|_{L^2}\\
    \leq\, & C_0(\|\u\|_{W^{1,\infty}}\|  \nabla(\epsilon p-\theta)\|_{H^{s}}
    +\|\u\|_{H^{s+1}}\|\nabla(\epsilon p-\theta)\|_{L^{\infty}}).
\end{align*}
It follows from the definition of $\mathcal{Q}$ and Sobolev's inequalities that
\begin{align*}
 \|  \nabla(\epsilon p,\theta)\|_{H^{s}}\leq \mathcal{Q},\;\;\;\;\;\|\nabla(\epsilon
 p,\theta)\|_{L^\infty}\leq \mathcal{Q}.
\end{align*}
Therefore, we obtain $\|g_1\|_{L^2}\leq C(\mathcal{Q})(1+\mathcal{S})$, and
\begin{align*}
|\langle  p_\alpha,  g_1\rangle|\leq C(\mathcal{Q})(1+\mathcal{S}).
\end{align*}

Next, we turn to the term $|\langle\u_\alpha ,g_2\rangle|$.
  Due to the equation \eqref{hnabb}, one has
\begin{align}\label{slowbf}
 -[\partial^\alpha, b(-\theta)]\partial_t(\epsilon \u)
 =\,& [\partial^\alpha, b(-\theta)] \big((\u\cdot \nabla)(\epsilon \u)\big)
 + [\partial^\alpha, b(-\theta)]\big(b(\theta)\nabla p\big)\nonumber\\
& -[\partial^\alpha, b(-\theta)]\big(b(\theta) a(\epsilon p)(\cu
\H\times
 (\epsilon\H))\big)\nonumber\\
&-[\partial^\alpha, b(-\theta)]\big(b(\theta) a(\epsilon p)\dv
\Psi(\epsilon \u)\big).
\end{align}
The inequality \eqref{comc} implies that
\begin{align*}
& \left|\left\langle \u_\alpha,[\partial^\alpha, b(-\theta)]
\big((\u\cdot \nabla)(\epsilon \u)\big)\right\rangle\right| \nonumber \\
&\quad \leq C_0\| \u_\alpha\|_{L^2}\|[\partial^\alpha, b(-\theta)]
\big(\u\cdot \nabla(\epsilon \u)\big)\|_{L^2}\nonumber\\
& \quad\leq C(\mathcal{Q})\big(\|b(-\theta)\|_{W^{1,\infty}}\|\u\cdot
\nabla(\epsilon \u)\|_{H^s}+\|b(-\theta)\|_{H^{s+1}}\|\u\cdot
\nabla(\epsilon \u)\|_{L^\infty}\big)\nonumber\\
&\quad \leq C(\mathcal{Q}),
\end{align*}
and
\begin{align*}
& \left|\left\langle \u_\alpha,[\partial^\alpha, b(-\theta)]
\big(b(\theta)\nabla p\big)\right\rangle\right|  \nonumber \\
& \quad\leq C_0\| \u_\alpha\|_{L^2}\|[\partial^\alpha, b(-\theta)]
\big(b(\theta)\nabla p\big)\|_{L^2}\nonumber\\
&\quad \leq C(\mathcal{Q})\big(\|b(-\theta)\|_{W^{1,\infty}}\|b(\theta)\nabla
p\|_{H^s}+\|b(-\theta)\|_{H^{s+1}}\|b(\theta)\nabla p\|_{L^\infty}\big)\nonumber\\
&\quad \leq C(\mathcal{Q})(1+\mathcal{S}).
\end{align*}
The third term on  the right-hand side  of \eqref{slowbf} can be
treated in a similar manner, and we obtain
\begin{align*}
  \left|\left\langle \u_\alpha, [\partial^\alpha, b(-\theta)]\big(b(\theta)
  a(\epsilon p)(\cu \H\times (\epsilon\H))\big)\right\rangle\right|\leq C(\mathcal{Q})(1+\mathcal{S}).
\end{align*}

To bound the last term on the right-hand side of \eqref{slowbf}, we use
\eqref{comc} to deduce that
\begin{align*}
&  \left\langle \u_\alpha,[\partial^\alpha, b(-\theta)]\big(b(\theta)
a(\epsilon p)\dv \Psi(\epsilon \u)\big)\right\rangle\nonumber\\
& \quad\leq C_0\|\u_\alpha\|_{L^2}\|[\partial^\alpha,
b(-\theta)]\big(b(\theta) a(\epsilon p)\dv \Psi(\epsilon
\u)\big)\|_{L^2}\nonumber\\
&\quad \leq C(\mathcal{Q})(\|b(-\theta)\|_{W^{1,\infty}}\|b(\theta)
a(\epsilon p)\dv \Psi(\epsilon \u)\|_{H^s}  \\
&\quad\quad +\|b(-\theta)\|_{H^{s+1}}\|b(\theta) a(\epsilon p)\dv
\Psi(\epsilon \u)\|_{L^\infty})\nonumber\\
& \quad\leq C(\mathcal{Q})(1+\mathcal{S}). \end{align*}
Hence, it holds that
\begin{align*}
\left|\left\langle  \u_\alpha,  g_2\right\rangle\right| \leq
C(\mathcal{Q})(1+\mathcal{S}).
\end{align*}

Since $g_3$ is similar to $g_1$ in structure, we easily get
\begin{align*}
\big|\big\langle \hat\theta_\alpha,  g_3\big\rangle\big|\leq
C(\mathcal{Q})(1+\mathcal{S}).
\end{align*}

Putting all estimates above into \eqref{3.21}, we conclude from    the positivity of
$b(-\theta)$, and the definitions of $\mathcal{O}$, $\mathcal{O}_0$,
$\mathcal{Q}$ and $\mathcal{S}$, that there exists a
constant $l_1>0$, such that for $t\in [0,T]$ and $T=\min\{T_1,1\}$, it holds that
\begin{align*}
&  \sup_{\tau \in [0,t]}\|(\hat{p}_\alpha, \hat\u_\alpha,
\hat \theta_\alpha )(\tau)\|^2_{L^2}
    +l_1\int^t_0  \|  \nabla(\hat\u_\alpha, \hat\theta_\alpha )(\tau)\|^2_{L^2}{\rm d}\tau\nonumber\\
&\quad \leq  C(\mathcal{O}_0)+C(\mathcal{O}(t))t +C(\mathcal{O}(t))\int^t_0\mathcal{S}(\tau){\rm d} \tau\nonumber\\
&\quad \leq  C(\mathcal{O}_0)+C(\mathcal{O}(t))\sqrt{t} \nonumber\\
&\quad \leq  C(\mathcal{O}_0)\exp\{\sqrt{T} C(\mathcal{O}(T)\}.
\end{align*}

Summing up the above estimates for all $\alpha$ with
$0\leq |\alpha|\leq s+1$, we obtain the desired inequality \eqref{slowao}.
\end{proof}
In a way similar to the proof of Lemma \ref{LLb}, we can show that
\begin{lem}\label{LLb1}
 Let $s\geq 4$  be an integer and $(p,\u,\H,\theta)$ be a solution to \eqref{hnaaa}--\eqref{hnadd},
 \eqref{hnas} on $[0,T_1]$.
 Then there exist a constant $l_2>0$ and an increasing function $C(\cdot)$ such that,
 for any $\epsilon \in (0,1]$ and $t\in [0,T]$, $T=\min\{T_1,1\}$, it holds that
\begin{align}\nonumber 
 & \sup_{\tau\in [0,t]}\|(\epsilon^2 q,\epsilon^2\u, \epsilon(\theta-\bar \theta)(\tau)\|_{H^{s+2}}\nonumber\\
 & \qquad   +l_2\bigg \{\int^t_0\|\nabla (\epsilon^2\u,\epsilon\theta)(\tau)\|^2_{H^{s+2}}{\rm d}\tau\bigg\}^{1/2}
 \leq C(\mathcal{O}_0)\exp\{(\sqrt{T})C(\mathcal{O}(T))\}.\nonumber
 \end{align}
\end{lem}

Recalling Lemma \ref{Lc} and the definitions of $\mathcal{Q}$ and $\mathcal{S}$, we find that
\begin{align}
& \|\partial_t(\epsilon p,\epsilon
\u,\epsilon\H,\theta)\|_{H^{s-1}}\leq
C(\mathcal{Q}),\label{t11}\\
&\|\partial_t(\epsilon p,\epsilon
\u,\epsilon\H,\theta)\|_{H^{s}}\leq C(\mathcal{Q})(1+\mathcal{S}),\label{t111} \\
&
\epsilon\|\partial_t(\epsilon p,\epsilon
\u,\epsilon\H,\theta)\|_{H^{s}}\leq C(\mathcal{Q}). \label{t12}
\end{align}

Moreover, it follows easily from Lemmas \ref{HES}--\ref{LLb1} and the equation
\eqref{hnadd} that for some constant $l_3>0$,
\begin{align}
\sup_{\tau\in [0,t]}\|(\epsilon\partial_t\theta)(\tau)\|_{H^s}^2+
l_3\int_0^t\|\nabla(\epsilon\partial_t\theta)(\tau)\|_{H^s}^2{\rm d}\tau\leq
C(\mathcal{O}_0) \exp\big\{\sqrt{T}   C(\mathcal{O}(T)\big\}.  \label{theta1}
\end{align}

\subsection{$H^{s-1}$-estimates on $(\dv \u,\nabla p)$}

To establish the estimates for $p$ and the acoustic part of $\u$, we
first control the term $(\epsilon\partial_t)(p,\u)$. To this end, we start with
an $L^2$-estimate for the linearized system.
For a given state $(p_0,\u_0,\H_0,\theta_0)$, consider
the following linearized system of \eqref{hnaaa}--\eqref{hnadd}:
\begin{align}
   &\partial_t p  +(\u_0  \cdot\nabla)p
   +\frac{1}{\epsilon}\dv(2\u -\bar \kappa  a(\epsilon p_0)b(\theta_0)\nabla \theta ) =\epsilon a(\epsilon p_0)
[\bar \nu \, \cu\H_0 :\cu\H]\nonumber\\
& \quad \quad \qquad \quad\quad \ \
+\epsilon a(\epsilon p_0)\Psi(\u_0 ):\nabla\u  +\bar \kappa a(\epsilon p_0)
b(\theta_0)\nabla p_0  \cdot \nabla  \theta +f_1, \label{slowa} \\
&b(-\theta_0)[\partial_t\u +(\u_0 \cdot\nabla)\u ]+\frac{\nabla p }{\epsilon}
  =a( \epsilon p_0)[(\cu \H_0)\times \H+\dv\Psi (\u )]+f_2, \label{slowb} \\
&\partial_t\H -\cu(\u_0 \times\H )-\bar\nu \Delta\H =f_3,\quad
\dv\H =0,\label{slowc} \\
&\partial_t \theta  +(\u_0 \cdot\nabla)\theta + \dv\u  =\epsilon^2 a(\epsilon p_0)
[\bar \nu\, \cu\H_0 :\cu\H
+\Psi (\u_0 ):\nabla\u ]\nonumber\\
& \quad \quad \qquad \quad\qquad \qquad\quad \ \ +\bar \kappa  a(\epsilon p_0)\dv (b(\theta_0)\nabla  \theta )+f_4,
  \label{slowd}
\end{align}
where we have added the source terms $f_i$ ($1\leq i\leq 4$) on the right-hands sides
of \eqref{slowa}--\eqref{slowd} for later use, and used the following notations:
$$ a(\epsilon p_0): = e^{-\epsilon p_0}, \quad b(\theta_0): =e^{\theta_0}.$$
The system \eqref{slowa}--\eqref{slowd} is supplemented with initial data
\begin{align}
(p ,\u ,\H , \theta )|_{t=0}
=(p_{\rm in}(x),\u_{\rm in}(x),\H_{\rm in}(x), \theta_{\rm in}(x)), \quad x \in \mathbb{R}^3. \label{slowe}
\end{align}

\begin{lem}\label{Lefast}
Let $( p , \u , \H ,  \theta )$ be a solution to the Cauchy problem \eqref{slowa}--\eqref{slowe}
on $[0,\hat T]$. Then there exist a constant $l_4>0$ and an increasing function $C(\cdot)$ such that,
 for any $\epsilon \in (0,1]$ and $t\in [0,T]$, $T=\min\{\hat T,1\}$, it holds that
  \begin{align}\label{fasta}
  & \sup_{\tau \in [0,t]}\|( p ,  \u ,  \H)(\tau)\|^2_{L^2}
    +l_4\int^t_0  \|  \nabla(  \u , \H)(\tau)\|_{L^2}^2{\rm d}\tau\nonumber\\
 &\quad \leq e^{TC(R_0)}\|(  p ,  \u ,  \H )(0)\|^2_{L^2}
 +C(R_0)e^{TC(R_0)}\sup_{\tau \in [0,T]}\|\nabla\theta(\tau)\|^2_{L^2}\nonumber\\
 &\quad\quad +C(R_0)\int_0^T\|\nabla(\epsilon\u,\epsilon\H)(\tau)\|_{L^2}^2{\rm d}\tau+
C(R_0) \int^T_0  \|\nabla \theta(\tau)\|^2_{H^1}{\rm d} \tau\nonumber\\
& \quad\quad + C(R_0)\int^T_0\left\{\|f_1\|^2_{L^2}+\|f_2\|^2_{L^2}+\|f_3\|^2_{L^2}+ \|\nabla
f_4\|_{L^2}^2\right\}(\tau){\rm d}\tau ,
     \end{align}
 where
\begin{align}\label{R00}
 R_0=\sup_{\tau\in [0,T]}\big\{\|\partial_t\theta_0(\tau)\|_{L^\infty},
 \|(p_0, \u_0,\H_0, \theta_0)(\tau)\|_{W^{1,\infty}}\big\}.
 \end{align}
\end{lem}

\begin{proof}
Set
 \begin{align*}
 (\tilde p,\tilde \u, \tilde \H,\tilde \theta)
 =(p, 2\u-\bar \kappa a(\epsilon p_0)b(\theta_0)\nabla \theta, \H,\theta).
\end{align*}
 Then $\tilde p$ and $\tilde \H$ satisfy
\begin{align}
\partial_t \tilde  p
 &+(\u_0  \cdot\nabla)\tilde p
   +\frac{1}{\epsilon}\dv \tilde\u   = \epsilon a(\epsilon p_0)
[\bar \nu \, \cu\H_0 :\cu\tilde \H]
+\frac{\epsilon}{2} a(\epsilon p_0)\Psi(\u_0 ):\nabla \tilde\u\nonumber\\
&   + \frac{\epsilon}{2} a(\epsilon p_0)\Psi(\u_0 ):\nabla(\bar
\kappa a(\epsilon p_0)b(\theta_0)\nabla\tilde \theta) +\bar \kappa
a(\epsilon p_0)b(\theta_0)\nabla p_0  \cdot \nabla \tilde \theta
+f_1 \label{fastb}
\end{align}
and
\begin{align}
  &\partial_t\tilde \H -\cu(\u_0 \times\tilde \H)-\bar\nu \Delta\tilde \H =  f_3,\quad
\dv\tilde \H =0, \label{fastbbb}
\end{align}
respectively. One can derive the equation for $\tilde \theta$ by
applying the operator $\nabla$ to \eqref{slowd} to obtain that
\begin{align}  \label{fastc}
  & \partial_t \nabla\tilde  \theta  +(\u_0 \cdot\nabla)\nabla \tilde \theta + \frac12 \nabla \dv\tilde \u
    + \frac12 \nabla \dv\big(\bar \kappa a(\epsilon p_0)b(\theta_0)\nabla  \tilde \theta\big)\nonumber\\
 & \quad= \nabla \big\{\epsilon^2 a(\epsilon p_0)
[\bar \nu\, \cu\H_0 :\cu\tilde\H]\big\}
+\frac12 \nabla\big\{\epsilon^2 a(\epsilon p_0)\Psi (\u_0 ):\nabla\tilde\u \big\}\nonumber\\
&\quad\quad +\frac12 \nabla\big\{\epsilon^2 a(\epsilon p_0)
\Psi (\u_0 ):\nabla(\bar \kappa a(\epsilon p_0)b(\theta_0)\nabla \tilde \theta) \big\}\nonumber\\
&\quad \quad +\nabla\big\{\bar \kappa  a(\epsilon p_0)\dv (b(\theta_0)\nabla
\tilde \theta )\big\} +[\nabla,\u_0]\cdot \nabla \tilde \theta
+\nabla f_4.
\end{align}

Multiplying \eqref{fastc} with $\frac{1 }{2}\bar \kappa a(\epsilon p_0)$, we get
\begin{align}  \label{fastd}
& \frac12 b(-\theta_0)\big\{\partial_t( \bar \kappa a(\epsilon
p_0)b(\theta_0) \nabla\tilde  \theta)
 +(\u_0 \cdot\nabla)[\bar \kappa a(\epsilon p_0)b(\theta_0)\nabla \tilde \theta]\big\}
 \nonumber\\
& \quad= \frac{\bar \kappa}{2} b(-\theta_0)\partial_t\{a(\epsilon p_0)
b(\theta_0)\} \nabla\tilde  \theta
+\frac{\bar \kappa}{2} b(-\theta_0)\big\{\u_0\cdot \nabla
[a(\epsilon p_0)b(\theta_0)] \nabla\tilde  \theta\big\}\nonumber\\
& \quad\quad - \frac{1 }{2} \bar \kappa a(\epsilon p_0)  \nabla \dv \tilde \u
  - \frac{1 }{2} \bar \kappa a(\epsilon p_0)  \nabla \dv
  \big(\bar \kappa a(\epsilon p_0)b(\theta_0)\nabla \tilde \theta\big)\nonumber\\
& \quad\quad  +\frac{1 }{2} \bar \kappa a(\epsilon p_0)
\nabla\big\{\epsilon^2 a(\epsilon p_0) [\bar \nu\, \cu\H_0
:\cu\tilde \H]\big\}\nonumber\\
& \quad\quad  +\frac{1 }{4} \bar \kappa a(\epsilon p_0)  \nabla\big\{\epsilon^2 a(\epsilon p_0)
  \Psi (\u_0 ):\nabla \tilde\u  \big\} \nonumber\\
& \quad\quad +\frac{1 }{4} \bar \kappa a(\epsilon p_0)  \nabla\big\{\epsilon^2
a(\epsilon p_0)
  \Psi (\u_0 ):\nabla(\bar \kappa a(\epsilon p_0)b(\theta_0)\nabla \tilde\theta)\big\}\nonumber\\
& \quad\quad +\frac{1 }{2} \bar \kappa a(\epsilon p_0)
\nabla\big\{\bar \kappa  a(\epsilon p_0)\dv (b(\theta_0)\nabla \tilde \theta )\big\}\nonumber\\
&\quad\quad +\frac{1 }{2} \bar \kappa a(\epsilon p_0) [\nabla,\u_0]\cdot \nabla
\tilde \theta
 +\frac{1 }{2} \bar \kappa a(\epsilon p_0)  \nabla f_4.
\end{align}
Subtracting \eqref{fastd} from \eqref{slowb} yields
\begin{align}\label{faste}
\frac{1}{2}  b(-\theta_0)&[\partial_t\tilde \u+\u_0\cdot \nabla \tilde\u]
+\frac{\nabla \tilde p}{\epsilon}\nonumber\\
=\, & -\frac{\bar \kappa}{2} b(-\theta_0)\partial_t\{a(\epsilon
p_0)b(\theta_0)\} \nabla\tilde  \theta
-\frac{\bar \kappa}{2} b(-\theta_0)\big\{\u_0\cdot \nabla [a(\epsilon p_0)b(\theta_0)]
\nabla\tilde  \theta\big\}\nonumber\\
 &+ \frac{1 }{4} \bar \kappa a(\epsilon p_0)  \nabla \dv \tilde \u
 + \frac{1 }{4} \bar \kappa a(\epsilon p_0)  \nabla \dv (
  \bar \kappa  a(\epsilon p_0)b(\theta_0) \nabla\tilde \theta ) \nonumber\\
&-\frac{1 }{2} \bar \kappa a(\epsilon p_0)  \nabla\big\{\epsilon^2
a(\epsilon p_0)
[\bar \nu\, \cu\H_0 :\cu\tilde \H]\big\}\nonumber\\
&-\frac{1 }{4} \bar \kappa a(\epsilon p_0)  \nabla\big\{\epsilon^2
a(\epsilon p_0) \Psi (\u_0 ):\nabla\tilde \u \big\}\nonumber\\
& -\frac{1 }{4} \bar \kappa a(\epsilon p_0)  \nabla\big\{
\epsilon^2 a(\epsilon p_0)\Psi (\u_0 ):\nabla(
   \bar \kappa  a(\epsilon p_0)b(\theta_0) \nabla\tilde \theta ) \big\}\nonumber\\
&   -\frac{1 }{2} \bar \kappa a(\epsilon p_0)
   \nabla\big\{\bar \kappa  a(\epsilon p_0)\dv (b(\theta_0)
\nabla\tilde\theta)\big\}+ a(\epsilon p_0) [\bar \nu\, \cu\H_0 :\cu\tilde \H]\nonumber\\
 &   +   \frac{1}{2}a(\epsilon p_0) [\Psi (\u_0 ):\nabla\tilde \u ]
 + \frac{1}{2}a(\epsilon p_0) [\Psi (\u_0 ):\nabla (  \bar \kappa
 a(\epsilon p_0)b(\theta_0) \nabla\tilde \theta ) ]\nonumber\\
& -\frac{1 }{2} \bar \kappa a(\epsilon p_0) [\nabla,\u_0]\cdot
\nabla \tilde \theta +a( \epsilon p_0)[(\cu \H_0)\times \tilde\H] \nonumber\\
 & +\frac{1}{2} a( \epsilon p_0)\dv\Psi (\bar \kappa a(\epsilon p_0)b(\theta_0)\nabla \tilde\theta)
 +\frac{1}{2} a( \epsilon p_0)\dv\Psi (\tilde \u )-\frac{1 }{2}
  \bar \kappa a(\epsilon p_0)  \nabla f_4+f_2\nonumber\\
:=\,& \sum^{14}_{i=1}h_i  +\frac{1}{2} a( \epsilon p_0)\dv\Psi (\tilde\u )
 -\frac{1 }{2} \bar \kappa a(\epsilon p_0)  \nabla f_4+f_2.
\end{align}
 Multiplying \eqref{fastb} by $\tilde p$, \eqref{fastbbb} by $\tilde \H$, and
\eqref{faste} by $\tilde \u$ respectively, integrating over  $\mathbb{R}^3$,
and summing up the resulting equations, we deduce that
\begin{align}\label{fastf}
\frac{\rm d}{{\rm d}t} \Big\{\frac12 &\langle  \tilde   p,  \tilde   p\rangle
+\frac14 \langle b(-\theta_0)  \tilde   \u,  \tilde   \u\rangle
+\frac12 \langle  \tilde \H,\tilde\H \rangle \Big\}+ \bar \nu \|\nabla
\tilde\H\|^2_{L^2}
 \nonumber\\
= & -\langle  (\u_0\cdot \nabla) \tilde   p,  \tilde   p \rangle
+\frac{1}{4}\langle \partial_t b(-\theta_0)  \tilde   \u, \tilde
\u\rangle
-\frac12 \langle   b(-\theta_0)(\u_0\cdot \nabla) \tilde   \u,  \tilde   \u\rangle\nonumber\\
&+\langle   \epsilon a(\epsilon p_0)
[\bar \nu \, \cu\H_0 :\cu\tilde \H], \tilde p\rangle +\frac{\epsilon}{2}\langle
a(\epsilon p_0)\Psi(\u_0 ):\nabla \tilde\u, \tilde p\rangle\nonumber\\
&   + \frac{\epsilon}{2}\langle a^2(\epsilon
p_0)b(\theta_0)\Psi(\u_0 ):\nabla(\nabla \tilde\theta), \tilde
p\rangle
+\langle \bar \kappa a(\epsilon p_0)b(\theta_0)\nabla p_0  \cdot \nabla \tilde \theta, \tilde p\rangle\nonumber\\
& +\sum^{14}_{i=1}\left\langle h_i, \tilde   \u \right\rangle
+\frac12 \langle a(\epsilon p_0) \dv \Psi(\tilde\u),\tilde\u\rangle\nonumber\\
&  - \frac{1 }{2}\left \langle \bar \kappa a(\epsilon p_0)  \nabla
f_4, \tilde \u \right\rangle +\left\langle f_2, \tilde \u
\right\rangle + \langle f_3, \tilde \H  \rangle +\langle f_1,\tilde p\rangle,
\end{align}
where the singular terms have been canceled out.

 Now, the terms on the right-hand side of \eqref{fastf} can be estimated as follows.
First, it follows from the regularity of $(p_0,\u_0,\H_0,\theta_0)$, a partial integration
and Cauchy-Schwarz's inequality that
 \begin{align}
 &
\frac{1}{4}\left|\langle \partial_t b(-\theta_0) \tilde \u,\tilde
\u\rangle\right| \leq \frac{1}{4}\| \partial_t b(-\theta_0)\|_{L^\infty}\|\tilde
 \u\|^2_{L^2}\leq C(R_0)\|\tilde \u\|^2_{L^2},\nonumber\\ 
&|\langle  (\u_0\cdot \nabla)\tilde p, \tilde p \rangle|
=\frac{1}{2}\left|\int(\dv \u_0)|\tilde p|^2dx\right|
\leq  C(R_0)\|\tilde p\|_{L^2}^2,\nonumber\\ 
&\frac12 |\langle   b(-\theta_0)(\u_0\cdot \nabla)\tilde \u, \tilde
\u\rangle| \leq  C(R_0)\|\tilde \u\|_{L^2}^2,\nonumber\\ 
& |\langle   \epsilon a(\epsilon p_0)
[\bar \nu \, \cu\H_0 :\cu\tilde \H], \tilde p\rangle|
\leq C(R_0)(\|\epsilon\nabla \tilde\H\|^2_{L^2}+\|\tilde p\|^2_{L^2}),\nonumber\\ 
& \frac{\epsilon}{2}|\langle a(\epsilon p_0)\Psi(\u_0 ):\nabla
\tilde\u, \tilde p\rangle|
\leq C(R_0)(\|\epsilon \nabla \tilde\u\|^2_{L^2}+\|\tilde p\|^2_{L^2}), \nonumber\\ 
& \frac{\epsilon}{2}|\langle a^2(\epsilon p_0)b(\theta_0)\Psi(\u_0 )
:\nabla(\nabla \tilde\theta), \tilde p\rangle|
 \leq C(R_0)\|\tilde p\|^2_{L^2}   +G_1(\epsilon p_0,\theta_0)\sum_{|\alpha|=2}
\|\partial^\alpha(\epsilon  \tilde\theta)\|^2_{L^2}, \nonumber \\
& |\langle \bar \kappa a(\epsilon p_0)b(\theta_0)\nabla p_0  \cdot
\nabla \tilde \theta, \tilde p\rangle| \leq   C(R_0)(\| \nabla
\tilde\theta\|^2_{L^2}+\|\tilde p\|^2_{L^2}),\nonumber 
\end{align}
where $G_1(\cdot,\cdot)$ is a smooth function. Similarly, one can bound the
terms involving $h_i$ in (\ref{faste}) as follows.
\begin{align} 
\sum^{14}_{i=1} |\langle h_i, \tilde \u\rangle| \leq
& \frac{\bar
\nu}{8}\|\nabla\tilde \H\|^2_{L_2}+\frac{\underline a \bar
\mu}{8}\|\nabla\tilde \u\|^2_{L_2}
+\frac{\underline a \bar \nu}{8}\|\dv\tilde \u\|^2_{L_2}\nonumber\\
&+ C(R_0) \|\tilde \u\|^2_{L_2}+ C(R_0) \|\nabla (\epsilon \tilde
\u,\tilde \theta)\|^2_{L_2} + G_2(\epsilon p_0,\theta_0)\|\Delta\tilde \theta\|^2_{L_2},\nonumber
\end{align}
where $G_2(\cdot,\cdot)$ is a smooth function.

For the dissipative term
$\frac12 \langle a(\epsilon p_0) \dv \Psi( \tilde \u),\tilde\u\rangle $,
we can employ arguments similar to those used in the estimate of the slow motion in
\eqref{slowt}--\eqref{slowv} to obtain that
\begin{align}\nonumber 
-\frac12\langle a(\epsilon p_0) \dv \Psi(\hat \u), \hat
\u\rangle\geq \frac{\underline a \bar \mu}{4}  (\|\nabla \hat
\u\|^2_{L^2}+  \|\dv \hat \u\|^2_{L^2})
 -C(R_0)\|\hat \u\|^2_{L^2}.
\end{align}

  Finally, putting all estimates above into \eqref{fastf} and applying
Cauchy-Schwarz's and Gronwall's inequalities, we get \eqref{fasta}.
\end{proof}

In the next lemma we utilize Lemma \ref{Lefast} to control
$\big((\epsilon \partial_t)p, (\epsilon \partial_t)\u,(\epsilon \partial_t) \H\big)$.
\begin{lem}\label{LLfast}
Let $s\geq 4$  be an integer and $(p,\u,\H,\theta)$ be the solution to the Cauchy problem
\eqref{hnaaa}--\eqref{hnadd}, \eqref{hnas} on $[0,T_1]$. Set
 $$(p_\beta, \u_\beta,\H_\beta,\theta_\beta):=\partial^\beta\big((\epsilon\partial_t)p,
 (\epsilon\partial_t)\u,(\epsilon\partial_t)\H ,(\epsilon\partial_t)\theta\big),  $$
where $0\leq |\beta|\leq s-1$. Then there exist a constant $l_5>0$ and an increasing function $C(\cdot)$
such that, for any $\epsilon \in (0,1]$ and $t\in [0,T]$, $T=\min\{ T_1,1\}$, it holds that
\begin{align} \label{fastba}
&\sup_{\tau \in [0,t]}\|(p_\beta, \u_\beta,\H_\beta)(\tau)\|^2_{L^2}
\nonumber\\
&\qquad + l_5\int^t_0  \|  \nabla( \u_\beta,\H_\beta)(\tau)\|_{L^2}^2{\rm d}\tau \leq C(\mathcal{O}_0)\exp\big\{\sqrt{T} C(\mathcal{O}(T))\big\}.
\end{align}
\end{lem}
\begin{proof} An application of the operator $\partial^\beta(\epsilon\partial_t)\,(0\leq |\beta|\leq s-1)$ to
the system \eqref{hnaaa}--\eqref{hnadd} leads to
\begin{align}
   &\partial_t p_\beta  +(\u  \cdot\nabla)p_\beta
   +\frac{1}{\epsilon}\dv\big(2\u_\beta -\bar \kappa  a(\epsilon p )b(\theta )\nabla \theta_\beta \big)
   =\epsilon a(\epsilon p )[\bar \nu \, \cu\H :\cu\H_\beta] \nonumber\\
& \quad \quad  +\epsilon a(\epsilon p )\Psi(\u ):\nabla\u_\beta
 +\bar \kappa a(\epsilon p )b(\theta )\nabla p  \cdot \nabla  \theta_\beta + \tilde g_1, \label{fastbb} \\
&b(-\theta )[\partial_t\u_\beta +(\u \cdot\nabla)\u_\beta ]+\frac{\nabla p_\beta }{\epsilon}
  =a( \epsilon p )[(\cu \H )\times \H_\beta +\dv\Psi (\u_\beta )]+\tilde g_2, \label{fastbc} \\
&\partial_t\H_\beta -\cu(\u \times\H_\beta )-\bar\nu \Delta\H_\beta =\tilde g_3,\quad
\dv\H_\beta =0,\label{fastbd} \\
&\partial_t \theta_\beta  +(\u \cdot\nabla)\theta_\beta + \dv\u_\beta =\epsilon^2 a(\epsilon p )
[\bar \nu\, \cu\H :\cu\H_\beta]\nonumber\\
& \quad \quad
+\epsilon^2 a(\epsilon p )\Psi (\u ):\nabla\u_\beta   +\bar \kappa  a(\epsilon p )
\dv (b(\theta )\nabla  \theta_\beta )+\tilde g_4,    \label{fastbe}
\end{align}
where
\begin{align*}
\tilde g_1:=\,&  -[\partial^\beta(\epsilon \partial_t), \u]\cdot \nabla  p
+\frac{1}{\epsilon}[\partial^\beta(\epsilon \partial_t),(\bar \kappa  a(\epsilon p )
 b(\theta ))] \Delta  \theta \nonumber\\
 & +\frac{1}{\epsilon} [\partial^\beta(\epsilon \partial_t), \nabla (\bar \kappa
 a(\epsilon p )  b(\theta ))]\cdot \nabla \theta
 +\epsilon\bar \nu\,[ \partial^\beta(\epsilon \partial_t), a(\epsilon p )
 \cu\H]  :\cu\H\nonumber\\
& +\epsilon[\partial^\beta(\epsilon \partial_t), a(\epsilon p)\Psi (\u  )]:\nabla\u
+[\partial^\beta(\epsilon \partial_t),\bar \kappa a(\epsilon p )b(\theta )\nabla p]\cdot\nabla\theta,
           \\
\tilde g_2:=\,&  -[\partial^\beta(\epsilon \partial_t), b(-\theta)]\partial_t \u
       -[\partial^\beta(\epsilon \partial_t), b(-\theta)\u]\cdot \nabla   \u,\nonumber\\
    &  +[\partial^\beta(\epsilon \partial_t), a(\epsilon p)\cu \H]\times   \H
    +[\partial^\beta(\epsilon \partial_t), a(\epsilon p)]\dv \Psi( \u), \\
\tilde g_3:=\,& \,  \partial^\beta(\epsilon \partial_t)\big(\cu(\u\times  \H)\big)
-\cu(\u\times \H_\beta), \\
\tilde g_4:=\,& -[\partial^\beta(\epsilon \partial_t), \u]\cdot \nabla \theta
+\epsilon^2\bar \nu\,[ \partial^\beta(\epsilon \partial_t), a(\epsilon p )
 \cu\H]  :\cu \H \nonumber\\
& +\epsilon^2[\partial^\beta(\epsilon \partial_t), a(\epsilon p)\Psi (\u  )]:\nabla \u \nonumber\\
&  +\bar \kappa \partial^\beta(\epsilon \partial_t)\big ( a(\epsilon p )\dv (b(\theta )
\nabla   \theta_\alpha) \big)
-\bar \kappa a(\epsilon) \dv  (b(\theta )\nabla   \theta_\alpha ).    
\end{align*}
It follows from the linear estimate \eqref{fasta} that, for some
$l_4>0$,
    \begin{align}  \label{fastbq}
 \sup_{\tau \in [0,t]} &\|(p_\beta,  \u_\beta ,  \H_\beta)(\tau)\|^2_{L^2}
    +l_4\int^t_0  \|  \nabla( \u_\beta ,   \H_\beta)(\tau)\|_{L^2}^2{\rm d}\tau\nonumber\\
 & \leq   e^{TC(R)}\|(p_\beta, \tilde \u_\beta ,  \H_\beta)(0)\|^2_{L^2}
 +C(R)e^{TC(R)}\sup_{\tau \in [0,T]}\|\nabla\theta_\beta(\tau)\|^2_{L^2}\nonumber\\
& \quad +TC(R)\sup_{\tau \in [0,T]}\|\nabla(  \epsilon \u_\beta, \epsilon
\H_\beta)(\tau)\|^2_{L^2}+C(R)\int^T_0  \|\nabla  \theta_\beta(\tau)\|^2_{H^1}{\rm d}\tau \nonumber\\
& \quad + C(R)\int^T_0\left \{\|\tilde{g}_1\|^2_{L^2}+\|\tilde{g}_2\|^2_{L^2}
+\|\tilde{g}_3\|^2_{L^2} + \|\nabla \tilde{g}_4\|_{L^2}^2\right\}(\tau){\rm d}\tau  ,
\end{align}
 where $R$ is defined as $R_0$ in  \eqref{R00} with $(p_0,\u_0,\H_0,
 \theta_0)$ replaced with $(p,\u,\H,\theta)$.

Now we control the norms $\|\tilde g_1\|_{L^2}^2$, $\|\tilde g_2\|_{L^2}^2$,
$\|\tilde g_3\|_{L^2}^2$, and  $\|\nabla \tilde g_4\|_{L^2}^2$.
The first term of $\tilde g_1$ can be bounded as follows.
 \begin{align} 
\left\|[\partial^\beta(\epsilon \partial_t), \u]\cdot \nabla  p\right\|_{L^2}
\leq\, & \epsilon C_0 (\|\u\|_{H^{s-1}}\|(\epsilon \partial_t) \nabla p\|_{H^{s-2}}
    +\|(\epsilon\partial_t) \u\|_{H^{s-1}}\|\nabla p \|_{H^{s-1}})\nonumber\\
    \leq\, & C(\mathcal{Q}). \nonumber
\end{align}
Similarly, the second term of $\tilde g_1$ admits the following boundedness:
\begin{align}  
  & \frac{1}{\epsilon}\|[\partial^\beta(\epsilon \partial_t),(\bar \kappa
  a(\epsilon p )  b(\theta ))] \Delta  \theta\|\nonumber\\
& \quad \leq C_0 \big(\|  a(\epsilon p )  b(\theta )\|_{H^{s-1}}\|\partial_t\Delta \theta \|_{H^{s-2}}
+ \|\partial_t(  a(\epsilon p )  b(\theta ))\|_{H^{s-1}} \|\Delta  \theta\|_{H^{s-1}}\big)
\nonumber\\
& \quad \leq C(\mathcal{Q})(1+\mathcal{S}). \nonumber
\end{align}
The other four terms in $\tilde g_1$ can be treated similarly and hence
can be bounded from above by $C(\mathcal{Q})(1+\mathcal{S})$.

For the first term of $\tilde g_2$, one has by the equation \eqref{hnabb} that
\begin{align}  \label{fastbl}
 [\partial^\beta(\epsilon \partial_t), b(-\theta)]\partial_t \u=
 & [\partial^\beta(\epsilon \partial_t), b(-\theta)]\{(\u \cdot\nabla)\u\}\nonumber\\
 &+\frac{1}{\epsilon} [\partial^\beta(\epsilon \partial_t), b(-\theta)]\{b^{-1}(-\theta_0)\nabla p \}\nonumber\\
& - [\partial^\beta(\epsilon \partial_t), b(-\theta)]\{b^{-1}(-\theta)a( \epsilon p)[(\cu \H_0)\times \H]\}\nonumber\\
& - [\partial^\beta(\epsilon \partial_t), b(-\theta)]\{b^{-1}(-\theta)a( \epsilon p)\dv\Psi (\u ) \} .
\end{align}
Note that the terms on the right-hand side of \eqref{fastbl} have similar structure as that of $\tilde g_1$.
Thus, we see that
\begin{align} \nonumber
\|[\partial^\beta(\epsilon \partial_t), b(-\theta)]\partial_t \u\|_{L^2}\leq   C(\mathcal{Q})(1+\mathcal{S}).
\end{align}
Similarly, the other four terms of $\tilde g_2$ can be bounded from above by $C(\mathcal{Q})(1+\mathcal{S})$.

Next, by the identity \eqref{vb}, one can rewrite $\tilde g_3$ as
 \begin{align}\nonumber
\tilde g_3  
     =& -[\partial^\beta(\epsilon \partial_t), \dv \u] \H
     - [\partial^\beta(\epsilon \partial_t),\u]\cdot \nabla  \H
     +\sum_{i=1}^3 [\partial^\beta(\epsilon \partial_t), \nabla \u_i]   \H .
\end{align}
Following a process similar to that in the estimates of $\tilde g_1$, one gets
\begin{align} \nonumber
\|\tilde g_3\|_{L^2}\leq   C(\mathcal{Q})(1+\mathcal{S}).
\end{align}
And analogously,
\begin{align} \nonumber
\|\nabla \tilde g_4\|_{L^2}\leq   C(\mathcal{Q})(1+\mathcal{S}).
\end{align}

We proceed to control the other terms on the right-hand side of
\eqref{fastbq}. It follows from \eqref{theta1} that
$$ C(R)e^{TC(R)}\sup_{\tau \in
 [0,T]}\|\nabla\theta_\beta(\tau)\|^2_{L^2}\leq
 C(\mathcal{O}(T)) \exp\big\{\sqrt{T}   C(\mathcal{O}(T)\big\} $$
and
$$ \int^T_0  \|\Delta   \theta_\beta(\tau)\|^2_{L_2}{\rm d}\tau \leq
\int^T_0  \|(\epsilon \partial_t)  \theta (\tau)\|^2_{H^{s+1}}{\rm d}\tau
 \leq  C(\mathcal{O}_0) \exp\big\{\sqrt{T}  C(\mathcal{O}(T))\big\}. $$
Thanks to \eqref{t12}, one has
\begin{align}
 TC(R)\sup_{\tau \in [0,T]}\|\nabla(  \epsilon \u_\beta, \epsilon  \H_\beta)(\tau)\|^2_{L^2}
& \leq   TC(\mathcal{O}(T))\sup_{\tau \in [0,T]}\|(\epsilon
\partial_t)( \epsilon \u, \epsilon \H)(\tau)\|^2_{H^{s}} \nonumber \\
&\leq TC(\mathcal{O}(T)).\nonumber%
\end{align}
Then, the desired inequality \eqref{fastba} follows from the above estimates and
 the inequality \eqref{fastbq}.
\end{proof}

Now we are in a position to estimate the
Sobolev norm of $(\dv \u,\nabla p)$ based on Lemma \ref{LLfast}.
\begin{lem}\label{LLfastb}
Let $s\geq 4$  be an integer and $(p,\u,\H,\theta)$ be the solution to the Cauchy problem
\eqref{hnaaa}--\eqref{hnadd}, \eqref{hnas} on $[0,T_1]$.
Then there exist a constant $l_6>0$ and an increasing function $C(\cdot)$ such that,
 for any $\epsilon \in (0,1]$ and $t\in [0,T_1]$, $T=\min\{ T_1,1\}$, it holds that
\begin{align} \label{fastca}
\sup_{\tau \in [0,t]}\left\{\|p(\tau)\|_{H^s}+\|\dv \u(\tau)\|_{H^{s-1}}\right\}
& + l_6\int^t_0\left\{\|\nabla p\|^2_{H^s}+\|\nabla\dv\u\|^2_{H^{s-1}}\right\}(\tau){\rm d}\tau\nonumber\\
&  \leq C(\mathcal{O}_0)\exp\big\{(\sqrt{T}+\epsilon)C(\mathcal{O}(T))\big\}.
\end{align}
\end{lem}

\begin{proof}
Rewrite the equations \eqref{hnaaa} and \eqref{hnabb} as
\begin{align}
\dv \u = & -\frac{1}{2}(\epsilon \partial_t) p -\frac{\epsilon}{2}(\u  \cdot\nabla)p+
 \frac{1}{2}\dv(\bar \kappa  a(\epsilon p)b(\theta)\nabla \theta)+\frac{\epsilon^2\bar \nu }{2} a(\epsilon p_0)
|\cu\H|^2\nonumber\\
& +\frac{\epsilon^2}{2} a(\epsilon p)\Psi(\u):\nabla\u
 +\frac{\epsilon\bar \kappa}{2} a(\epsilon p)b(\theta)\nabla p  \cdot \nabla  \theta, \label{fastcb} \\
{\nabla p }= & -  b(-\theta) (\epsilon\partial_t)\u -{\epsilon} b(-\theta)(\u \cdot\nabla)\u \nonumber\\
 & +\epsilon a( \epsilon p)[(\cu \H)\times \H] +\epsilon a( \epsilon p)\dv\Psi (\u ). \label{fastcc}
\end{align}
Then,
\begin{align}
\|\dv \u\|_{H^{s-1}}\leq\,  &C_0\|(\epsilon \partial_t) p\|_{H^{s-1}}
+C_0 \epsilon\,\|\u\|_{H^{s-1}}\|\nabla p\|_{H^{s-1}}\nonumber\\
& +C_0 \|\dv(\bar \kappa  a(\epsilon p)b(\theta)\nabla \theta)\|_{H^{s-1}}+C_0 \| a(\epsilon p_0)\|_{L^\infty}
\|\epsilon\,\cu\H\|_{H^{s-1}}^2\nonumber\\
& +C_0\| a(\epsilon p)\|_{L^\infty}\|\Psi(\epsilon\u):(\epsilon\nabla\u)\|_{H^{s-1}} \nonumber\\
& +C_0\| a(\epsilon p)b(\theta)\|_{L^\infty}\|(\epsilon\nabla p)\|_{H^{s-1}}\| \nabla \theta\|_{H^{s-1}}.
\label{fastcd}
\end{align}
It follows from Lemmas \ref{EP1}--\ref{LLb1} and \ref{LLfast}, and the inequalities
\eqref{t11}--\eqref{theta1} that
\begin{align*}
\|(\epsilon \partial_t) p\|_{H^{s-1}}\leq\, &  C(\mathcal{O}_0)\exp\big\{(\sqrt{T}+\epsilon)C(\mathcal{O}(T))\big\},
\\
 \epsilon\,\|\u\|_{H^{s-1}}\|\nabla p\|_{H^{s-1}}\leq\, & \epsilon C(\mathcal{O}), \\
 \|\dv(\bar \kappa  a(\epsilon p)b(\theta)\nabla \theta)\|_{H^{s-1}}
  \leq\, & C_0 \|\Delta \theta\|_{H^{s-1}}+C_0\|\nabla \theta\|_{H^{s-1}}\\
   \leq\, & C(\mathcal{O}_0)\exp\big\{(\sqrt{T}+\epsilon)C(\mathcal{O}(T))\big\}, \\
\|\epsilon\,\cu\H\|_{H^{s-1}}\leq\, & C(\mathcal{O}_0)\exp\big\{(\sqrt{T}+\epsilon)C(\mathcal{O}(T))\big\}, \\
  \|\Psi(\epsilon\u):(\epsilon\nabla\u)\|_{H^{s-1}}
  \leq\, & C(\mathcal{O}_0)\exp\big\{(\sqrt{T}+\epsilon)C(\mathcal{O}(T))\big\}, \\
\|(\epsilon\nabla p)\|_{H^{s-1}}\| \nabla  \theta\|_{H^{s-1}}
\leq\, & C(\mathcal{O}_0)\exp\big\{(\sqrt{T}+\epsilon)C(\mathcal{O}(T))\big\}. \nonumber
\end{align*}
These bounds together with \eqref{fastcd} imply that
\begin{align}\nonumber 
\sup_{\tau \in [0,t]} \|\dv \u\|(\tau)\|_{H^{s-1}}
\leq   C(\mathcal{O}_0)\exp\big\{(\sqrt{T}+\epsilon)C(\mathcal{O}(T))\big\}.
\end{align}
Similar arguments applying to the equation \eqref{fastcc} for $\nabla p$ yield
\begin{align}\label{fastcf}
  \sup_{\tau \in [0,t]} \|p\|_{H^s}
   + l_6\int^t_0  \|\nabla p(\tau)\|^2_{H^s}{\rm d}\tau
\leq      C(\mathcal{O}_0)\exp\big\{(\sqrt{T}+\epsilon)C(\mathcal{O}(T))\big\}
\end{align}
for some positive constant $l_6>0$.

To obtain the desired inequality \eqref{fastca}, we shall establish the following estimate
 \begin{align} \label{fastcg}
   \int^T_0 \| \nabla \dv \u(\tau)\|^2_{H^{s-1}}{\rm d}\tau
\leq      C(\mathcal{O}_0)\exp\big\{(\sqrt{T}+\epsilon)C(\mathcal{O}(T))\big\}.
\end{align}
In fact, for any multi-index $\alpha$ satisfying $1\leq |\alpha|\leq s$, one can apply
the operator $\partial^\alpha$ to \eqref{fastcb} and then take the inner product
with $\partial^\alpha \dv \u$ to obtain
\begin{align}\label{fastch}
\int^T_0 \|\partial^\alpha \dv \u(\tau)\|^2_{L^2}{\rm d}\tau
= &  -\frac{1}{2}\int^T_0 \langle \partial^\alpha(\epsilon \partial_t) p,
\partial^\alpha \dv \u\rangle (\tau){\rm d}\tau\nonumber\\
& +\int^T_0\langle \Xi,\partial^\alpha \dv \u\rangle (\tau){\rm d}\tau,
\end{align}
where
\begin{align*}
\Xi : =& -\frac{\epsilon}{2}(\u  \cdot\nabla)p+
 \frac{1}{2}\dv(\bar \kappa  a(\epsilon p)b(\theta)\nabla \theta)+\frac{\epsilon^2\bar \nu }{2} a(\epsilon p_0)
|\cu\H|^2\nonumber\\
& +\frac{\epsilon^2}{2} a(\epsilon p)\Psi(\u):\nabla\u
 +\frac{\epsilon\bar \kappa}{2} a(\epsilon p)b(\theta)\nabla p  \cdot \nabla  \theta.
\end{align*}
It thus follows from \eqref{fastba} and similar arguments to those for \eqref{fastcf} that, for all $1\leq |\alpha|\leq s$,
\begin{align*}
\int^T_0 \|\partial^\alpha\Xi (\tau)\|^2_{L^2}{\rm d}\tau
\leq C(\mathcal{O}_0)\exp\big\{(\sqrt{T}+\epsilon)C(\mathcal{O}(T))\big\},
\end{align*}
whence,
\begin{align}
&\int^T_0\left|\langle \Xi,\partial^\alpha \dv \u\rangle\right| (\tau){\rm d}\tau\nonumber\\
& \quad\leq C(\mathcal{O}_0)\exp\big\{(\sqrt{T}+\epsilon)C(\mathcal{O}(T))\big\}
\bigg\{\int^T_0\| \partial^\alpha \dv \u\|^2_{L^2} (\tau){\rm d}\tau \bigg\}^{1/2} \nonumber\\
& \quad\leq C(\mathcal{O}_0)\exp\big\{(\sqrt{T}+\epsilon)C(\mathcal{O}(T))\big\}
+\frac14 \int^T_0\| \partial^\alpha \dv \u(\tau)\|^2_{L^2}{\rm d}\tau . \nonumber
\end{align}

For the first term on the right-hand side of \eqref{fastch}, one gets by integration by parts that
\begin{align*}
 -\frac{1}{2}\int^T_0 \langle \partial^\alpha(\epsilon \partial_t) p, \partial^\alpha \dv \u\rangle (\tau){\rm d}\tau
 = \,& -\frac{1}{2} \langle \partial^\alpha p, \epsilon\partial^\alpha \dv (\epsilon\u)\rangle\Big|^T_0 \nonumber\\
 & +\frac{1}{2}\int^T_0\langle \partial^\alpha\nabla p,\partial^\alpha (\epsilon\partial_t) \u\rangle (\tau){\rm d}\tau.
\end{align*}
By virtue of the estimate \eqref{slowao} on $(\epsilon q,\epsilon\u,\theta-\bar \theta)$ and \eqref{fastcf},
we find that
\begin{align}
\left| \frac{1}{2} \langle \partial^\alpha p, \epsilon\partial^\alpha \dv (\epsilon\u)\rangle\Big|^T_0\right|
\leq\, & \sup_{\tau\in [0,T]}\{\|p(\tau)\|_{H^s}\|\epsilon\u(\tau)\|_{H^{s+1}}\}\nonumber\\
\leq\, & C(\mathcal{O}_0)\exp\big\{(\sqrt{T}+\epsilon)C(\mathcal{O}(T))\big\},  \nonumber\\
  \frac{1}{2} \bigg|\int^T_0 \langle \partial^\alpha \nabla p,
  \partial^\alpha (\epsilon \partial_t)  \u\rangle (\tau){\rm d}\tau\bigg|
\leq \, &\frac{1}{2} \bigg\{\int^T_0 \|\partial^\alpha\nabla p (\tau)\|^2_{L^2}{\rm d}\tau\bigg\}^{1/2}\nonumber\\
 & \times \bigg\{\int^T_0 \|\partial^\alpha (\epsilon \partial_t)  \u(\tau)\|^2_{L^2}{\rm d}\tau\bigg\}^{1/2}
 \nonumber\\
\leq\, & C(\mathcal{O}_0)\exp\big\{(\sqrt{T}+\epsilon)C(\mathcal{O}(T))\big\}. \nonumber
\end{align}
These bounds, together with \eqref{fastch}, yield the desired estimate \eqref{fastcg} after summing $\alpha$ over  $1\leq |\alpha|\leq s$.
 This completes the proof.
 \end{proof}

\subsection{$H^{s-1}$-estimate on $\cu \u$ }

The another key point to obtain a uniform bound for $\u$ is the following estimate on $\cu\u$.
\begin{lem}\label{LLT}
Let $s\geq 4$ be an integer and  $(p,\u,\H,\theta)$ be the solution to the Cauchy problem
\eqref{hnaaa}--\eqref{hnadd}, \eqref{hnas} on $[0,T_1]$.
Then there exist a constant $l_7>0$ and an increasing function $C(\cdot)$ such that,
 for any $\epsilon \in (0,1]$ and $t\in [0,T_1]$, $T=\min\{ T_1,1\}$, it holds that
\begin{align} \label{exta}
& \sup_{\tau \in [0,t]}\left\{\|\cu(b({-\theta})\u)(\tau)\|^2_{H^{s-1}}
+\|\cu \H (\tau)\|^2_{H^{s-1}}\right\}\nonumber\\
& \quad + l_7\int^t_0\left\{\|\nabla\cu(b({-\theta})\u)\|^2_{H^{s-1}}
+\|\nabla\cu \H (\tau)\|^2_{H^{s-1}}\right\}(\tau){\rm d}\tau  \nonumber\\
& \qquad  \leq   C(\mathcal{O}_0)\exp\big\{(\sqrt{T}+\epsilon)C(\mathcal{O}(T))\big\}.
\end{align}
\end{lem}
\begin{proof} Applying the operator \emph{curl} to the equations \eqref{hnabb} and \eqref {hnacc},
using the identities \eqref{vbd} and \eqref{vbc}, and the fact that $\cu \nabla =0$, one infers that
\begin{align}
& \partial_t (\cu(b(-\theta)\u)) +(\u \cdot\nabla)(\cu(b(-\theta)\u))\nonumber\\
 & \quad  \  \  = \cu\{a( \epsilon p) (\cu \H)\times \H\}
 +\bar\mu \dv\{a(\epsilon p)b(\theta)\nabla (\cu(b(-\theta)\u))\}+ \Upsilon_1, \label{extb} \\
&\partial_t(\cu \H) -\cu[\cu(\u\times\H )]-\bar\nu \Delta(\cu\H) = 0,\label{extc}
\end{align}
where $\Upsilon_1$ is defined by
\begin{align*}
\Upsilon_1:=\, &\bar \mu  \dv\big(a(\epsilon p)(\nabla b(\theta))\otimes \cu(b(-\theta)\u)\big)
-\bar \mu \nabla a(\epsilon p)\cdot \nabla (b(\theta)\cu(b(-\theta)\u))\nonumber\\
& - \bar \mu a(\epsilon p)\Delta((\nabla b(\theta))\times (b(-\theta)\u))
-\nabla a(\epsilon p)\times ( \bar \mu \Delta \u+(\bar \mu +\bar \lambda)\nabla \dv \u)\nonumber\\
& +\cu(b(-\theta)\u \partial_t \theta)+[\cu, \u]\cdot \nabla (b(-\theta)\u)
+\cu(b(-\theta)\u(\u\cdot\nabla \theta)).
\end{align*}

For any multi-index $\alpha$ satisfying $0\leq |\alpha|\leq s-1$, we apply
the operator $\partial^\alpha$ to \eqref{extb} and \eqref{extc} to obtain
\begin{align}
&   \partial_t \partial^\alpha (\cu(b(-\theta)\u)) +(\u \cdot\nabla)[\partial^\alpha(\cu(b(-\theta)\u))]\nonumber\\
 & \quad \quad   \  \  =  \partial^\alpha \cu\{a( \epsilon p) (\cu \H)\times \H\}\nonumber\\
 & \quad  \quad\quad  \  \   +\bar\mu \dv\{a(\epsilon p)b(\theta)\nabla [ \partial^\alpha (\cu(b(-\theta)\u))]\}
 +  \partial^\alpha \Upsilon_1 + \Upsilon_2  , \label{extd} \\
&  \partial_t \partial^\alpha (\cu \H) -\partial^\alpha\cu[\cu(\u\times\H )]-\bar\nu \Delta(\cu\H) = 0,\label{exte}
\end{align}
where
\begin{align*}
 \Upsilon_2: = & -[\partial^\alpha, \u]\cdot \nabla[\partial^\alpha(\cu(b(-\theta)\u))]\nonumber\\
 & -[\partial^\alpha, \dv (a(\epsilon p)b(\theta))]\nabla[\partial^\alpha(\cu(b(-\theta)\u))].
\end{align*}
Multiplying \eqref{extd} by $\partial^\alpha (\cu(b(-\theta)\u))$ and \eqref{exte} by  $\partial^\alpha (\cu \H)$
 respectively, summing up, and integrating  over $\mathbb{R}^3$, we deduce that
\begin{align}\label{extf}
\frac{1}{2}\frac{\rm d}{{\rm d}t} &\{\|\partial^\alpha (\cu(b(-\theta)\u))\|^2_{L^2}
+\|\partial^\alpha (\cu \H)\|^2_{L^2} \} +\bar \nu \|\partial^\alpha (\cu \H)\|^2_{L^2}\nonumber\\
&
+ \langle a(\epsilon p)b(\theta)\nabla [ \partial^\alpha (\cu(b(-\theta)\u))],\nabla
[ \partial^\alpha (\cu(b(-\theta)\u))]  \rangle \nonumber\\
=\,&-\langle (\u \cdot\nabla)[\partial^\alpha(\cu(b(-\theta)\u))],
\partial^\alpha (\cu(b(-\theta)\u))\rangle  \nonumber\\
& - \langle  \partial^\alpha \cu\{a( \epsilon p) (\cu \H)\times \H\},
\partial^\alpha (\cu(b(-\theta)\u))\rangle \nonumber\\
& + \langle \partial^\alpha\cu[\cu(\u\times\H )],  \partial^\alpha (\cu \H)\rangle
+\langle  \partial^\alpha \Upsilon_1 + \Upsilon_2,
\partial^\alpha (\cu(b(-\theta)\u))\rangle \nonumber\\
:=\, &  \mathcal{J}_1+ \mathcal{J}_2+ \mathcal{J}_3+ \mathcal{J}_4,
\end{align}
where $\mathcal{J}_i$ ($i=1,\cdots,4$) will be bounded as follows.

An integration by parts leads to
\begin{align*}
|\mathcal{J}_1|\leq \|\dv\u\|_{L^\infty}\|\nabla [ \partial^\alpha (\cu(b(-\theta)\u))]\|^2_{L^2}.
\end{align*}
By virtue of \eqref{va}, the Cauchy-Schwarz's inequality and Lemma \ref{Lb},
 the term $\mathcal{J}_2$ can be bounded as follows.
\begin{align*}
|\mathcal{J}_2 |\leq\, & | \langle  \partial^\alpha\{a( \epsilon p) (\cu \H)\times \H\},
\partial^\alpha \cu (\cu(b(-\theta)\u))\rangle|\nonumber\\
\leq \,& |\partial^\alpha\{a( \epsilon p) (\cu \H)\times \H\}\|_{L^2}
\|\partial^\alpha \nabla (\cu(b(-\theta)\u))\|_{L^2}\nonumber\\
\leq\,&  \eta_1\|\partial^\alpha \nabla (\cu(b(-\theta)\u))\|_{L^2}^2\nonumber\\
& + C\{\|\cu \H\|^2_{L^\infty}\|a( \epsilon p)\H\|_{H^{s-1}}^2
+\|a( \epsilon p)\H\|^2_{L^\infty}\|\cu \H\|^2_{H^{s-1}}\},
\end{align*}
where $\eta_{1}>0$ is a sufficiently small constant independent
of $\epsilon$.

If we integrate by parts, make use of \eqref{va} and the fact that
$\cu \cu \mathbf{a} =\nabla \,\dv \,\mathbf{a} -\Delta \mathbf{a}$ and
$\dv \H=0$, we see that the term  $\mathcal{J}_3$ can be rewritten as
\begin{align*}
\mathcal{J}_3   =   \left\langle  \partial^\alpha   \cu(\u\times \H),
  \partial^\alpha\Delta \H\right  \rangle ,\nonumber
 \end{align*}
which, together with the Moser-type inequality, implies that
\begin{align*}
|  \mathcal{J}_3|\leq C(\mathcal{S})+\eta_{2} \|\H^\epsilon(\tau)\|^2_{s+1},
\end{align*}
where $\eta_{2}>0$ is a sufficiently small constant independent of $\epsilon$.

 To handle $\mathcal{J}_4$, we note that the leading order terms in $\Upsilon_1$ are of third-order
  in $\theta$  and of
second-order in $\u$, and the leading order terms in $\Upsilon_2$ are of order $s+1$ in $\u$
and of order $s + 1$ in $(\epsilon p, \theta)$. Then it follows that
\begin{align*}
|\mathcal{J}_4| \leq\, & C_0 (  \|\partial^\alpha \Upsilon_1\|_{L^2}
+ \|\Upsilon_2\|_{L^2})\|  \partial^\alpha (\cu(b(-\theta)\u))\|_{L^2} \\
 \leq\,& C(\mathcal{S})\|  \partial^\alpha (\cu(b(-\theta)\u))\|_{L^2}.
 \end{align*}

Putting the above estimates into the \eqref{extf}, choosing
$\eta_1$ and $\eta_2$ sufficient small, summing over $\alpha$ for $0\leq |\alpha|\leq s-1$,
and then integrating the result on $[0,t]$,  we conclude
\begin{align*}
\sup_{\tau \in [0,t]}& \left\{\|\cu(b({-\theta})\u)(\tau)\|^2_{H^{s-1}}
+\|\cu \H (\tau)\|^2_{H^{s-1}}\right\}\nonumber\\
& \quad + l_7\int^t_0\left\{\|\nabla\cu(b({-\theta})\u)\|^2_{H^{s-1}}
+\|\nabla\cu \H \|^2_{H^{s-1}}\right\}(\tau){\rm d}\tau  \nonumber\\
& \leq C_0 \big \{\|\cu(b({-\theta})\u)(0)\|^2_{H^{s-1}}+\|\cu \H (0)\|^2_{H^{s-1}}\big\}
\nonumber\\
& \quad + C(\mathcal{O}_0)\exp\big\{(\sqrt{T}+\epsilon)C(\mathcal{O}(T))\big\}\nonumber\\
& \leq  C(\mathcal{O}_0)\exp\big\{(\sqrt{T}+\epsilon)C(\mathcal{O}(T))\big\}.
\end{align*}
\end{proof}

\begin{proof} [Proof of  Proposition \ref{hPb}]
Proposition \ref{hPb} follows directly from Lemmas  \ref{HES}, \ref{LLb}, \ref{LLfastb} and \ref{LLT},  and some arguments in
the proof of Lemma 6.28 in \cite{A06}.
\end{proof}

Once Proposition \ref{hPb} is established, the existence part of Theorem \ref{main} can be proved
 by directly applying the same arguments as in \cite{A06,MS01}, and hence we omit the details here.


\section{Decay of the local energy and zero Mach number limit}

In this section, we shall prove the convergence part of Theorem \ref{main} by modifying the
arguments developed by M\'{e}tivier and Schochet \cite{MS01}, see also
some extensions in \cite{A05,A06,LST}.

\begin{proof}[Proof of the convergence part of Theorem \ref{main}]
 The uniform estimate \eqref{conda} implies that
 \begin{align*}
 \sup_{\tau\in [0,T_0]} \|(p^\epsilon,\u^\epsilon,\H^\epsilon)(\tau)\|_{H^{s}}
 +  \sup_{\tau\in [0,T_0]} \| (\theta^\epsilon-\bar\theta)(\tau)\|_{H^{s+1}}<+\infty.
 \end{align*}
 Thus,  after extracting a subsequence, one has
   \begin{align}
      & (p^\epsilon, \u^\epsilon) \rightharpoonup (\bar p, \w )
    &  \text{weakly-}\ast \ \text{in} & \quad\quad L^\infty(0,T_0; H^s(\mathbb{R}^3)),
\label{heata}\\
 &   \H^\epsilon  \rightharpoonup  \t
   &   \text{weakly-}\ast \ \text{in} & \qquad L^\infty(0,T_0; H^s(\mathbb{R}^3)),
\label{heataa}\\
   & \theta^\epsilon-\bar\theta   \rightharpoonup   \vartheta-\bar\theta   \!\!\!\!\!\!
    &  \text{weakly-}\ast \ \text{in} & \qquad  L^\infty(0,T_0; H^{s+1}(\mathbb{R}^3)).
\label{heatb}
\end{align}
 It follows from the equations for $\H^\epsilon$ and $\theta^\epsilon$ that
 \begin{align}\label{heatbb}
   \partial_t \H^\epsilon,\, \partial_t\theta^\epsilon \in C([0,T_0],H^{s-2}(\mathbb{R}^3)).
 \end{align}
 \eqref{heataa}--\eqref{heatbb} implies, after
  further extracting a subsequence, that  for all $s'<s$,
\begin{align}
& \H^\epsilon  \rightarrow \t  \!\!\!\!\!\!\!\!\!\!\!\!\!  &  \text{strongly in}
  & \quad C([0,T_0],H^{s'}_{\mathrm{loc}}(\mathbb{R}^3)),\label{heatc}\\
   &       \theta^\epsilon -\bar \theta \rightarrow \vartheta -\bar\theta \!\!\!\!\!\!\!\!\!\!\!\!\!\!\!\!\!\!\!\!\!\!\!\! \!\!\!\!\!\!\!
    &
   \text{strongly in}&  \quad  C([0,T_0],
         H^{s'+1}_{\mathrm{loc}}(\mathbb{R}^3)), \label{heatd}
\end{align}
where the limits $\t \in C([0,T_0],
H^{s'}_{\mathrm{loc}}(\mathbb{R}^3))\cap
L^\infty(0,T_0;H^{s}_{\mathrm{loc}}(\mathbb{R}^3))$
and $\vartheta -\bar \theta \in C([0,T_0],\linebreak
H^{s'+1}_{\mathrm{loc}}(\mathbb{R}^3))\cap L^\infty(0,T_0;H^{s+1}_{\mathrm{loc}}(\mathbb{R}^3))$.

Similarly, from \eqref{exta} we get
\begin{align}
 &  \cu \big(e^{ -\theta^\epsilon} \u^\epsilon\big) \rightarrow \cu \big(e^{-\vartheta} \w \big)
 \quad \text{strongly in}  \quad    C([0,T_0],H^{s'-1}_{\mathrm{loc}}(\mathbb{R}^3))
 \label{heate}
  \end{align}
for all $ s'<s$.

In order to obtain the limit system, one needs to show that the
limits in \eqref{heata} hold in the strong topology of
$L^2(0,T_0;H^{s'}_{\mathrm{loc}}(\mathbb{R}^3))$ for all $s'<s$. To this end,
we first show that $\bar p=0$ and $\dv(2\w - \bar \kappa  e^\vartheta\nabla \vartheta)=0$.
In fact, the equations \eqref{hnaaa} and \eqref{hnabb} can be rewritten as
\begin{align}
& \epsilon \, \partial_t p^\epsilon+
\dv(2\u^\epsilon-\bar \kappa   e^{-\epsilon p^\epsilon+\theta^\epsilon}\nabla \theta^\epsilon)
= \epsilon f^\epsilon,\label{heatf}\\
&\epsilon\, e^{-\theta^\epsilon} \partial_t\u^\epsilon+ \nabla p^\epsilon
  = \epsilon\, \mathbf{g}^\epsilon. \label{heatg}
\end{align}
By virtue of \eqref{conda},  $f^\epsilon$ and  $\mathbf{g}^\epsilon$ are uniformly bounded
in $C([0,T_0],H^{s-1}(\mathbb{R}^3))$.
Passing to the weak limit in \eqref{heatf} and \eqref{heatg}, respectively, we see
that $\nabla \bar p=0$ and $\dv(2\w -\bar \kappa e^\vartheta\nabla \vartheta)=0$. Since $\bar p\in
L^\infty (0,T_0;H^s(\mathbb{R}^3))$, we infer that $\bar p=0$.

Notice that by virtue of \eqref{heate},
the strong compactness for the incompressible component of $e^{ -\theta^\epsilon} \u^\epsilon$ holds.
So, it is sufficient to prove the following proposition on the
acoustic components in order to get the strong convergence of  $\u^\epsilon$.

\begin{prop}\label{LC}
Suppose that the assumptions in Theorem \ref{main} hold. Then,
$p^\epsilon$ converges to $0$ strongly in $L^2(0,T_0;
H^{s'}_{\mathrm{loc}}(\mathbb{R}^3))$
 and $\dv(2\u^\epsilon-\bar \kappa   e^{-\epsilon p^\epsilon+\theta^\epsilon}\nabla \theta^\epsilon)$
converges to $0$ strongly in $L^2(0,T_0;
H^{s'-1}_{\mathrm{loc}}(\mathbb{R}^3))$ for all $s'<s$.
\end{prop}

The proof of Proposition \ref{LC} is based on the following dispersive estimates
on the wave equation obtained by M\'{e}tivier and Schochet \cite{MS01}
and reformulated in \cite{A06}.
\begin{lem} {\rm (\cite{MS01,A06})}\label{LD}
   Let $T>0$ and $v^\epsilon$ be a bounded sequence in $C([0,T],H^2(\mathbb{R}^3))$, such that
   \begin{align*}
\epsilon^2\partial_t(a^\epsilon \partial_t v^\epsilon)-\nabla\cdot (b^\epsilon \nabla v^\epsilon)=c^\epsilon,
   \end{align*}
where $c^\epsilon$ converges to $0$ strongly in $L^2(0,T;L^2(\mathbb{R}^3))$.
Assume further that for some $s> 3/2+1$, the
coefficients $(a^\epsilon,b^\epsilon)$ are uniformly bounded in
$C([0,T],H^s(\mathbb{R}^3))$ and converge in
$C([0,T],H^s_{\mathrm{loc}}(\mathbb{R}^3))$ to a limit $(a,b)$
satisfying the decay estimates
\begin{gather*}
   |a(x,t)-\hat a|\leq C_0 |x|^{-1-\zeta}, \quad |\nabla_x a(x,t)|\leq C_0 |x|^{-2-\zeta}, \\
 |b(x,t)-\hat b|\leq C_0 |x|^{-1-\zeta}, \quad |\nabla_x b(x,t)|\leq C_0 |x|^{-2-\zeta},
\end{gather*}
for some positive constants $\hat a$, $\hat b$, $C_0$
and $\zeta$. Then the sequence $v^\epsilon$ converges to $0$
strongly in $L^2(0,T; L^2_{\mathrm{loc}}(\mathbb{R}^3))$.
\end{lem}

\begin{proof}[Proof of Proposition \ref{LC}]
We fist show that $p^\epsilon$ converges to $0$ strongly  in
$L^2(0,T_0; \linebreak H^{s'}_{\mathrm{loc}}(\mathbb{R}^3))$ for all $s'<s$. Applying
$\epsilon^2 \partial_t$ to \eqref{hnaaa}, we find that
\begin{align}  \label{heath}
 &\epsilon^2\partial_t \{\partial_t p^\epsilon +(\u^\epsilon \cdot\nabla)p^\epsilon\}
   + {\epsilon}\partial_t\big\{\dv(2\u^\epsilon-\bar \kappa
   e^{-\epsilon p^\epsilon+\theta^\epsilon}\nabla \theta^\epsilon)\big\}\nonumber\\
 & \quad     =\epsilon^3\partial_t\big\{ e^{-\epsilon p^\epsilon}
[\bar\nu |\cu\H^\epsilon|^2
+\Psi(\u^\epsilon):\nabla\u^\epsilon]\big\}
+\epsilon^2 \partial_t\big\{\bar \kappa  e^{-\epsilon p^\epsilon
+\theta^\epsilon}\nabla p^\epsilon \cdot \nabla  \theta^\epsilon\big\}.
\end{align}
Dividing \eqref{hnabb} by $e^{-\theta^\epsilon}$
and then applying the operator \emph{div}  to the resulting equations, one gets
\begin{align} \label{heati}
\epsilon\partial_t\dv\u^\epsilon+ \dv \big(e^{\theta^\epsilon} {\nabla p^\epsilon}\big)
= & -\epsilon \dv\{(\u^\epsilon\cdot\nabla)\u^\epsilon\}  \nonumber\\
 &  + \epsilon \dv\big\{ e^{- \epsilon p^\epsilon+\theta^\epsilon}
 [(\cu \H)\times \H+\dv\Psi^\epsilon(\u^\epsilon)]\big\}.
\end{align}
Subtracting \eqref{heati} from \eqref{heath}, we have
\begin{align}\label{heatj}
\epsilon^2\partial_t \Big(\frac12\partial_t p^\epsilon\Big)
  -\dv \big(e^{\theta^\epsilon} {\nabla p^\epsilon}\big)
  = \epsilon F^\epsilon(p^\epsilon, \u^\epsilon, \H^\epsilon,\theta^\epsilon),
\end{align}
where $F^\epsilon(p^\epsilon, \u^\epsilon, \H^\epsilon,\theta^\epsilon)$ is a
smooth function in its variables with $F(0)=0$.
By the uniform boundedness of $(p^\epsilon, \u^\epsilon, \H^\epsilon,\theta^\epsilon)$
one infers
 that
\begin{align*}
\epsilon F^\epsilon(p^\epsilon, \u^\epsilon, \H^\epsilon,\theta^\epsilon) \rightarrow 0
\quad \text{strongly in} \quad L^2(0,T_0; L^2 (\mathbb{R}^3)).
\end{align*}

By the  strong convergence of $\theta^\epsilon$, the  initial conditions \eqref{dacay},
and the arguments in Section  8.1 in \cite{A06},  one can easily prove that
the coefficients in \eqref{heatj} satisfy the conditions in
Lemma \ref{LD}.
Therefore, we can apply Lemma \ref{LD} to obtain
\begin{align*}
p^\epsilon \rightarrow 0 \quad \text{strongly in} \quad
L^2(0,T_0; L^2_{\mathrm{loc}}(\mathbb{R}^3)).
\end{align*}
Since $p^\epsilon$ is bounded uniformly in $C([0,T_0], H^s(\mathbb{R}^3))$,
an interpolation argument gives
\begin{align*}
p^\epsilon \rightarrow 0 \quad \text{strongly in}
 \quad    L^2(0,T_0;  H^{s'}_{\mathrm{loc}}(\mathbb{R}^3))\ \ \text{for all} \ \  s'<s.
\end{align*}

Similarly, we can obtain the strong convergence of  $\dv(2\u^\epsilon
-\kappa^\epsilon e^{-\epsilon p^\epsilon+\theta^\epsilon}\nabla \theta^\epsilon)$.
This completes the proof.
\end{proof}

We continue our proof of Theorem \ref{main}. It follows from Proposition \ref{LC} and \eqref{heatd} that
\begin{align*}
\dv\,\u^\epsilon \rightarrow \dv\,\w\quad \text{strongly in}
\quad L^2(0,T_0; H^{s'-1}_{\mathrm{loc}}(\mathbb{R}^3)).
\end{align*}
Thus, using \eqref{heate}, one obtains
\begin{align*}
 \u^\epsilon \rightarrow \w\quad \text{strongly in} \quad L^2(0,T_0;
H^{s'}_{\mathrm{loc}}(\mathbb{R}^3))\qquad\mbox{for all }s'<s.
\end{align*}
By \eqref{heatc}, \eqref{heatd}, and Proposition \ref{LC}, we find that
\begin{equation*}
\begin{array}{lcl}
\nabla \u^\epsilon \rightarrow \nabla  \w    & \text{strongly in}\quad
&  L^2(0,T_0; H^{s'-1}_{\mathrm{loc}}(\mathbb{R}^3)),\\
\nabla \H^\epsilon \rightarrow \nabla  \t \qquad  & \text{strongly in}\quad &
L^2(0,T_0; H^{s'-1}_{\mathrm{loc}}(\mathbb{R}^3)),\\
\nabla \theta^\epsilon \rightarrow \nabla  \vartheta    & \text{strongly in}\quad  &
L^2(0,T_0; H^{s'-1}_{\mathrm{loc}}(\mathbb{R}^3)).
\end{array}
\end{equation*}
Passing to the limits in the equations for $p^\epsilon$, $\H^\epsilon$, and
$\theta^\epsilon$, respectively, one sees that the limit $(0, \w,\t, \vartheta)$ satisfies,
in the sense of distributions, that
\begin{align}
&\dv(2\w -\bar{\kappa}\, e^\vartheta\nabla \vartheta)=0, \label{heatk}   \\
&\partial_t\t -\cu(\w\times\t)-\bar \nu\Delta\t=0,\quad
\dv\t=0,    \label{heatl} \\
&\partial_t\vartheta  +(\w\cdot\nabla)\vartheta +\dv \w
=\bar \kappa  \, \dv (e^{\vartheta}\nabla  \vartheta).\label{heatm}
\end{align}

On the other hand, applying the operator \emph{curl} to the momentum equations \eqref{hnabb},
using the equations \eqref{hnaaa} and \eqref{hnadd}
on $p^\epsilon$ and $\theta^\epsilon$,
and then taking to the limit on the resulting equations, we deduce that
\begin{align*}
\cu\!\big\{\partial_t \big(e^{-\vartheta} \w)+\dv\big(\w e^{-\vartheta}
\otimes \w\big)-(\cu \t)\times \t-\dv\Phi(\w)\big\}=0
\end{align*}
holds in the sense of distributions. Therefore it follows from \eqref{heatk}--\eqref{heatm} that
\begin{align}
e^{-\vartheta}\{\partial_t\w+(\w\cdot\nabla)\w\}+\nabla \pi
  =(\cu \t)\times \t+\dv\Phi(\w), \label{heatn}
  \end{align}
for some function $\pi$.

Following the same arguments as those in the proof of
Theorem 1.5 in \cite{MS01}, we conclude that $(\w, \t, \vartheta)$
satisfies the initial condition
\begin{align}\label{heato}
 (\w,\t, \vartheta)|_{t=0} =(\w_0,\t_0, \vartheta_0),
\end{align}
where $\w_0$ is determined by
\begin{align}\nonumber
   \dv(2\w_0 -\bar{\kappa}\, e^{\vartheta_0}\nabla \vartheta_0)=0,
   \quad \cu(e^{-\vartheta_0}\w_0)= \cu(e^{-\vartheta_0}\u_0).
\end{align}
Moreover, the standard
iterative method shows that the system \eqref{heatk}--\eqref{heatn} with
initial data  \eqref{heato} has a unique solution $(\w^*, \t^*, \vartheta^*-\bar \theta)\in
C([0,T_0],H^s(\mathbb{R}^3))$. Thus, the uniqueness of solutions
to the limit system \eqref{heatk}--\eqref{heatn} implies that
the above convergence holds for
the full sequence of $(p^\epsilon, \u^\epsilon, \H^\epsilon, \theta^\epsilon)$.
Therefore the proof is completed.
\end{proof}

\bigskip \noindent
{\bf Acknowledgements:}
The authors are very grateful to the anonymous  referees for their constructive
comments and helpful suggestions.
This work was partially done when Li was visiting the Institute
of Mathematical Sciences, CUHK during the summer of 2011. He would like to thank the institute for hospitality.
Jiang is supported in part by the National Basic Research Program (Grant Nos. 2011CB309705, 2014CB745000), NSFC (Grant Nos. 11229101, 11371065), and Beijing Center for Mathematics and Information Interdisciplinary Sciences.  Ju is supported in part by NSFC (GrantNo.11171035) and BJNSF (Grant No. 1142001).  Li is supported in part by NSFC (Grant Nos. 11271184, 10971094), NCET-11-0227, PAPD, and the Fundamental Research Funds for the Central Universities. Xin is supported in part by Zheng Ge Ru Foundation, Hong Kong RGC Earmarked Research Grants CUHK-4041/11P and CUHK-4048/13P, The Focused Investment Scheme-Scheme B at The Chinese University of Hong Kong, and CAS-Croucher Funding Scheme for Joint Laboratories.




\end{document}